\documentclass[11pt,reqno]{amsart}

\usepackage{amssymb, amsmath, amsthm}
\usepackage{hyperref}
\usepackage[alphabetic,lite]{amsrefs}
\usepackage{verbatim}
\usepackage{amscd}   
\usepackage[all]{xy} 
\usepackage{youngtab} 
\usepackage{young} 
\usepackage{tikz}
\usepackage{ mathrsfs }
\usepackage{cases}
\usepackage{array}

\newcommand{\defi}[1]{{\upshape\sffamily #1}}

\renewcommand{\a}{\alpha}
\renewcommand{\b}{\beta}
\renewcommand{\d}{\delta}

\newcommand{\E}{\mathcal{E}}

\renewcommand{\det}{\textrm{det}}

\renewcommand{\ll}{\lambda}

\newcommand{\onto}{\twoheadrightarrow}
\newcommand{\oo}{\otimes}

\renewcommand{\P}{\mathcal{P}}

\renewcommand{\t}{\underline{t}}

\let\uuu\u 
\renewcommand{\u}{\underline{u}}

\renewcommand{\v}{\underline{v}}
\newcommand{\x}{\underline{x}}
\newcommand{\y}{\underline{y}}
\newcommand{\z}{\underline{z}}
\newcommand{\T}{\mathcal{T}}
\newcommand{\X}{\mathcal{X}}
\newcommand{\Y}{\mathcal{Y}}
\newcommand{\YU}{\mathcal{YU}}
\newcommand{\Z}{\mathcal{Z}}

\newcommand{\Ann}{\operatorname{Ann}}

\newcommand{\Ext}{\operatorname{Ext}}
\newcommand{\GL}{\operatorname{GL}}

\newcommand{\Sing}{\operatorname{Sing}}

\newcommand{\Sym}{\operatorname{Sym}}

\newcommand{\codim}{\operatorname{codim}}
\newcommand{\coker}{\operatorname{coker}}
\renewcommand{\det}{\operatorname{det}}
\newcommand{\dom}{\operatorname{dom}}

\renewcommand{\ker}{\operatorname{ker}}

\newcommand{\reg}{\operatorname{reg}}

\newcommand{\bb}[1]{\mathbb{#1}}

\renewcommand{\rm}[1]{\textrm{#1}}
\newcommand{\mc}[1]{\mathcal{#1}}
\newcommand{\mf}[1]{\mathfrak{#1}}

\newcommand{\ul}[1]{\underline{#1}}

\def\lra{\longrightarrow}

\def\llra{\longleftrightarrow}

\newcommand*{\lhra}{\ensuremath{\lhook\joinrel\relbar\joinrel\rightarrow}}

\newtheorem{theorem}{Theorem}[section]
\newtheorem*{theorem*}{Theorem}
\newtheorem*{problem*}{Problem}
\newtheorem{lemma}[theorem]{Lemma}

\newtheorem{proposition}[theorem]{Proposition}
\newtheorem{corollary}[theorem]{Corollary}
\newtheorem*{corollary*}{Corollary}

\newtheorem*{main-thm*}{Main Theorem}
\newtheorem*{linear-resolutions*}{Theorem on Linear Resolutions}
\newtheorem*{regularity-powers*}{Theorem on Regularity}
\newtheorem*{injectivity-Ext*}{Theorem on Injectivity of Maps of Ext Modules}
\newtheorem*{Kodaira*}{Kodaira Vanishing for Determinantal Thickenings}

\theoremstyle{definition}
\newtheorem{definition}[theorem]{Definition}
\newtheorem*{definition*}{Definition}
\newtheorem{example}[theorem]{Example}

\theoremstyle{remark}
\newtheorem{remark}[theorem]{Remark}
\newtheorem*{remark*}{Remark}

\numberwithin{equation}{section}

\begin{document}

\title{Regularity and cohomology of determinantal thickenings}

\author{Claudiu Raicu}
\address{Department of Mathematics, University of Notre Dame, 255 Hurley, Notre Dame, IN 46556\newline
\indent Institute of Mathematics ``Simion Stoilow'' of the Romanian Academy}
\email{craicu@nd.edu}

\subjclass[2010]{Primary 13D07, 14M12, 13D45}

\date{\today}

\keywords{Determinantal varieties, regularity, thickenings}

\begin{abstract} We consider the ring $S=\bb{C}[x_{ij}]$ of polynomial functions on the vector space $\bb{C}^{m\times n}$ of complex $m\times n$ matrices. We let $\GL=\GL_m(\bb{C})\times\GL_n(\bb{C})$ and consider its action via row and column operations on $\bb{C}^{m\times n}$ (and the induced action on $S$). For every $\GL$-invariant ideal $I\subseteq S$ and every $j\geq 0$, we describe the decomposition of the modules $\Ext^j_S(S/I,S)$ into irreducible $\GL$-representations. For any inclusion $I\supseteq J$ of $\GL$-invariant ideals we determine the kernels and cokernels of the induced maps $\Ext^j_S(S/I,S)\lra\Ext^j_S(S/J,S)$. As a consequence of our work, we give a formula for the regularity of the powers and symbolic powers of generic determinantal ideals, and in particular we determine which powers have a linear minimal free resolution. As another consequence, we characterize the $\GL$-invariant ideals $I\subseteq S$ for which the induced maps $\Ext^j_S(S/I,S) \lra H_I^j(S)$ are injective. In a different direction we verify that Kodaira vanishing, as described in work of Bhatt--Blickle--Lyubeznik--Singh--Zhang, holds for determinantal thickenings.
\end{abstract}

\dedicatory{Dedicated to David Eisenbud, on the occasion of his 70th birthday}

\maketitle

\section{Introduction}\label{sec:intro}

We consider positive integers $m\geq n$, let $X=(x_{ij})$ denote the generic $m\times n$ matrix, and let $S=\bb{C}[x_{ij}]$ denote the ring of polynomial functions on the space of complex $m\times n$ matrices. The ideal $I_p$ of $p\times p$ minors of $X$ defines the (projective) determinantal variety of matrices of rank smaller than $p$, while its powers $I_p^d$ define for $d>1$ \defi{thickenings} (non-reduced scheme structures) of the said variety. This article is concerned with the calculation of the modules $\Ext^{\bullet}_S(S/I,S)$ when $I$ defines an arbitrary equivariant thickening of a determinantal variety (where equivariance is considered with respect to the natural group action by $\GL=\GL_m(\bb{C})\times\GL_n(\bb{C})$). When $I\supseteq J$ define $\GL$-equivariant determinantal thickenings, we determine the kernel and cokernel of the induced map $\Ext^{\bullet}_S(S/I,S)\to\Ext^{\bullet}_S(S/J,S)$: the study of these maps was motivated by \cite{BBLSZ}. Translated via graded local duality, our results describe the cohomology groups $H^k(Y,\mc{O}_Y(j))$ for an arbitrary equivariant thickening $Y$ of a determinantal variety, and in particular they describe the Castelnuovo--Mumford regularity of such thickenings.

We begin with the following characterization of when powers of generic determinantal ideals have a linear minimal free resolution. This was already well-understood for maximal minors (see \cite[Thm.~5.4]{akin-buchsbaum-weyman} and \cite{bruns-conca-varbaro}), as well as for $1\times 1$ minors since they generate the maximal homogeneous ideal.

\begin{linear-resolutions*}\label{thm:linear-res}
 Consider positive integers $d$, and $m\geq n\geq p$, and let $I_p$ denote the ideal of $p\times p$ minors of the generic $m\times n$ matrix. The ideal $I_p^d$ has a linear resolution if and only if one of the following holds:
 \begin{enumerate}
  \item $p=1$.
  \item $p=n$.
  \item $p=2$ and $d\geq n-1$.
 \end{enumerate}
\end{linear-resolutions*}

Our contribution is to settle the intermediate cases ($1<p<n$) and in doing so we establish more precise results regarding the regularity of powers of determinantal ideals as follows. A celebrated result of Cutkosky--Herzog--Trung \cite[Thm.~1.1]{cutkosky-herzog-trung}, and independently of Kodiyalam \cite[Thm.~5]{kodiyalam}, asserts that for a homogeneous ideal $I$ the function $\reg(I^d)$ which measures the regularity of its powers is a linear function of $d$ when $d$ is sufficiently large. When $I=I_n$ is the ideal of maximal minors of the generic matrix, $\reg(I_n^d)=n\cdot d $ for all~$d$ by~\cite{akin-buchsbaum-weyman}. For the other determinantal ideals we prove the following (we write $I^{sat}$ for the saturation of an ideal $I$ with respect to the maximal homogeneous ideal, and $I^{(d)}$ for the $d$-th symbolic power of $I$).

\begin{regularity-powers*}\label{thm:regularity-detl}
 Consider positive integers $m\geq n>p\geq 1$, and let $I_p$ denote the ideal of $p\times p$ minors of the generic $m\times n$ matrix. We have that for $d\geq n-1$
 \begin{equation}\label{eq:reg-det-pows}
 \reg(I_p^d) = \reg((I_p^d)^{sat}) = p\cdot d + \begin{cases}
 \displaystyle \left(\frac{p-1}{2}\right)^2 & \rm{when }p\rm{ is odd}; \\
 \displaystyle \frac{(p-2)\cdot p}{4} & \rm{when }p\rm{ is even},
 \end{cases}
 \qquad\rm{ and }
 \end{equation}
 \begin{equation}\label{eq:reg-det-symb-pows}
 \reg(I_p^{(d)})=p\cdot d.
 \end{equation}
 When $p=2$, we have
 \begin{equation}\label{eq:reg-I2}
 \reg(I_2^d)=\reg(I_2^{(d)})=d+n-1\rm{ for }d=1,\cdots,n-2.
 \end{equation}
 When $p\geq 3$, we have 
 \begin{equation}\label{eq:reg-I>2}
\reg(I_p^d)\geq\reg(I_p^{(d)})>p\cdot d\rm{ for }d = 1,\cdots,n-2.
 \end{equation}
\end{regularity-powers*}

Since $I_p^d$ is generated in degree $p\cdot d$, (\ref{eq:reg-I>2}) implies that for $3\leq p\leq n-1$ the ideal $I_p^d$ doesn't have a linear resolution. For $p=2$, (\ref{eq:reg-det-pows}) and (\ref{eq:reg-I2}) show that $I_2^d$ has a linear resolution if and only if $d\geq n-1$. Together with the already known cases $p=1$ and $p=n$ (which can also be verified based on Lemmas~\ref{lem:Zsymb-vs-Zpow} and~\ref{lem:reg-I1-In}) this proves the Theorem on Linear Resolutions. We prove (\ref{eq:reg-det-pows}), (\ref{eq:reg-det-symb-pows}) and (\ref{eq:reg-I>2}) in Theorem~\ref{thm:reg-pows}, and (\ref{eq:reg-I2}) in Theorem~\ref{thm:2x2}.

The problem of finding effective bounds for the stabilization of regularity of powers and the question of determining the constant terms of the corresponding linear functions turn out to be quite subtle and have received a great deal of attention \cites{eisenbud-harris,eisenbud-ulrich,chardin}. The Theorem on Regularity provides, in the case of determinantal ideals, a concrete description of the constant terms, and it shows that the functions $\reg(I_2^d)$ for $n\geq 3$ and $\reg(I_p^{(d)})$ for $1<p<n$ stabilize to a linear function of $d$ precisely at $d=n-1$. For $p\geq 3$ the stabilization of $\reg(I_p^d)$ may occur earlier, but experiments suggest that our bound is close to being sharp (the smallest example of an early stabilization is for $n=9$, $p=7$, when stabilization occurs at $d=7$; we are aware of no example where stabilization occurs before $d=n-2$). 

Recall that the regularity of a graded $S$-module $M$ can be computed via \cite[Prop.~20.16]{eisenbud-CA}
\begin{equation}\label{eq:regFromExt}
\reg(M)=\max\{-r-j:\Ext^j_S(M,S)_r\neq 0\},
\end{equation}
and that $\reg(I) = \reg(S/I) + 1$ when $I$ is a homogeneous ideal, so the Theorem on Regularity can be proved using knowledge of the graded vector space structure of $\Ext^j_S(S/I,S)$ when $I=I_p^d$, $I=(I_p^d)^{sat}$, or $I=I_p^{(d)}$. Our main result completely describes this structure for an arbitrary $\GL$-invariant ideal $I$, and it can be stated slightly imprecisely as follows.
 
\begin{main-thm*}
 To any $\GL$-invariant ideal $I\subseteq S$ we can associate a finite set $\mc{M}(I)$ of $\GL$-equivariant $S$-modules with the property that for each $j\geq 0$
 \[\Ext^j_S(S/I,S) \simeq \bigoplus_{M\in\mc{M}(I)} \Ext^j_S(M,S),\]
 where the above isomorphism is $\GL$-equivariant and degree preserving (but in general it does not preserve the $S$-module structure). In particular, we get
 \[\reg(S/I) = \max_{M\in\mc{M}(I)}\reg(M).\]
 The sets $\mc{M}(I)$ and the modules $\Ext^j_S(M,S)$ for $M\in\mc{M}(I)$ can be computed explicitly. Furthermore, the association $I\mapsto\mc{M}(I)$ has the property that whenever  $I\supseteq J$ are $\GL$-invariant ideals, the (co)kernels and images of the induced maps $\Ext^j_S(S/I,S)\lra\Ext^j_S(S/J,S)$ can be computed as follows.
  \[\ker\left(\Ext^j_S(S/I,S)\lra\Ext^j_S(S/J,S)\right) = \bigoplus_{M\in\mc{M}(I)\setminus\mc{M}(J)}\Ext^{j}_S(M,S),\]
  \[\operatorname{Im}\left(\Ext^j_S(S/I,S)\lra\Ext^j_S(S/J,S)\right) = \bigoplus_{M\in\mc{M}(I)\cap\mc{M}(J)}\Ext^{j}_S(M,S),\]
  \[\coker\left(\Ext^j_S(S/I,S)\lra\Ext^j_S(S/J,S)\right) = \bigoplus_{M\in\mc{M}(J)\setminus\mc{M}(I)}\Ext^{j}_S(M,S).\]
 Finally, if we write $I:I_p^{\infty}$ for the saturation of $I$ with respect to $I_p$ then $\mc{M}(I:I_p^{\infty})\subseteq\mc{M}(I)$. More precisely
 \[\mc{M}(I:I_p^{\infty}) = \{M\in\mc{M}(I): \Ann(M) \not\subseteq I_p\}.\]
\end{main-thm*}

We make a couple of observations that will make the statement of the Main Theorem more precise, as well as provide some guidance to the reader:

\begin{itemize}
 \item The $\GL$-invariant ideals $I\subseteq S$ are indexed by sets of partitions and they have been classified in \cite{deconcini-eisenbud-procesi}: we recall the classification in Section~\ref{subsec:GLinv-ideals}.
 \item The modules $M$ appearing in $\mc{M}(I)$ are among the quotients $J_{\z,l}$ of $\GL$-invariant ideals defined in (\ref{eq:defJzl}) below. The calculation of the corresponding $\Ext$ modules was done in \cite[Section~3]{raicu-weyman} and is recalled in Theorem~\ref{thm:ExtJzl}.
 \item The set of pairs $(\z,l)$ for which $M=J_{\z,l}$ belongs to $\mc{M}(I)$ for a given $\GL$-invariant ideal $I$ is described in Definition~\ref{def:ZX}.
 \item A precise formulation of the Main Theorem is found in Theorem~\ref{thm:Ext-split-IX}, whose proof is the content of Section~\ref{sec:Ext-split-thick}.
 \item A concrete example of a calculation of $\Ext$ modules is presented in Section~\ref{sec:example}.
\end{itemize}

The Main Theorem yields two important instances where the induced maps on $\Ext$ modules are injective.

\begin{injectivity-Ext*} For $1\leq p\leq n$ and $d\geq 1$ consider the inclusions $I_p^d \subseteq I_p^{(d)}$ and $I_p^{(d+1)} \subseteq I_p^{(d)}$. For each $j\geq 0$ the induced maps on $\Ext$ modules
\[\Ext^j_S(S/I_p^{(d)},S) \lra \Ext^j_S(S/I_p^d,S) \quad \textrm{ and } \quad \Ext^j_S(S/I_p^{(d)},S) \lra \Ext^j_S(S/I_p^{(d+1)},S)\]
are injective.
\end{injectivity-Ext*}

The first injectivity follows from the fact that $I_p^{(d)} = I_p^d : I_{p-1}^{\infty}$, and therefore $\mc{M}(I_p^{(d)}) \subseteq \mc{M}(I_p^d)$ by the last assertion in the Main Theorem. The injectivity of $\Ext$ module maps induced by $I_p^{(d+1)} \subseteq I_p^{(d)}$ is proved in Corollary~\ref{cor:inj-Ext-symb}. It follows that for each $j\geq 0$ the local cohomology groups $H^j_{I_p}(S)$ can be described as a union of $\Ext$ modules
\[H^{j}_{I_p}(S) = \bigcup_{d\geq 1} \Ext^{j}_S(S/I_p^{(d)},S),\]
providing an answer in the case of determinantal ideals to a question of Eisenbud--Musta\c t\uuu a--Stillman \cite[Question~6.1]{EMS}. For a different answer, see \cite[Thm.~4.2 and Section~6]{raicu-weyman}. In \cite[Question~6.2 and Example~6.3]{EMS}, the authors show that in order for the induced maps $\Ext^{j}_S(S/I,S) \lra H^{j}_I(S)$ to be injective for all $j\geq 0$, the ideal $I$ needs to be unmixed. We prove in Corollary~\ref{cor:unmixed} that the converse holds when $I$ is a $\GL$-invariant ideal in $S=\bb{C}[x_{ij}]$.

The calculation of $\Ext$ modules in the Main Theorem can be translated via graded local duality to one of sheaf cohomology modules for the associated projective schemes. We obtain the following  vanishing theorem (see \cite[Section~3]{BBLSZ} and \cite{arapura-jaffe}), of which we prove a slightly stronger version in Section~\ref{sec:Kodaira}.

\begin{Kodaira*}
 For $m\geq n\geq p\geq 2$ we let $Y\subset\bb{P}^{m\cdot n-1}$ denote the projective variety of matrices of rank smaller than $p$, whose homogeneous ideal is $I_p$. For $d\geq 1$, we let $Y_d\subset\bb{P}^{m\cdot n-1}$ be the projective scheme defined by the ideal $I_p^d$ (the \defi{$d$-th thickening} of $Y$). We have that 
 \[H^k(Y_d,\mc{O}_{Y_d}(-j))=0\rm{ for }k<\codim(\Sing(Y))\rm{ and }j>0,\]
 with the convention that $\codim(\Sing(Y))=\dim(Y)$ when $Y$ is non-singular.
\end{Kodaira*}

\subsection*{Organization} In Section~\ref{subsec:GLinv-ideals} we recall the description and basic properties of $\GL$-invariant ideals from \cite{deconcini-eisenbud-procesi}, and in Section~\ref{subsec:ExtJzl} we recall from \cite{raicu-weyman} the calculation of the $\Ext$ modules for the subquotients $J_{\z,l}$ that are the building blocks for the coordinate rings of the determinantal thickenings. In Section~\ref{subsec:Ext-split-filtrations} we introduce the notion of $\Ext$-split filtrations, and in Section~\ref{sec:Ext-split-thick} we construct such filtrations for the coordinate rings of determinantal thickenings. In Section~\ref{sec:optimization} we set up and solve an optimization problem that will allow us to perform all the subsequent regularity calculations. In Section~\ref{sec:reg-pows} we determine the regularity of determinantal powers, while in Section~\ref{sec:Kodaira} we prove Kodaira vanishing for determinantal thickenings. We end with a concrete calculation of $\Ext$ modules in Section~\ref{sec:example}, explaining the table in \cite[Example~5.4]{BBLSZ}.


\section{Preliminaries}\label{sec:prelim}

\subsection{$\GL$-invariant ideals in $S$}\label{subsec:GLinv-ideals}

For $m\geq n$ we consider the ring $S=\Sym(\bb{C}^m\oo\bb{C}^n)$ of polynomial functions on the space of complex $m\times n$ matrices, with the natural action by $\GL=\GL_m(\bb{C})\times\GL_n(\bb{C})$. The $\GL$-invariant ideals $I\subset S$ have been completely classified in \cite{deconcini-eisenbud-procesi}. In order to recall their classification we need to set up some notation. A \defi{partition} $\ul{x}=(x_1,x_2,\cdots)$ is a finite collection of non-negative integers, with $x_1\geq x_2\geq\cdots$. We call each $x_i$ a \defi{part} of $\ul{x}$, and define the \defi{size} of $\ul{x}$ to be $|\ul{x}|=x_1+x_2+\cdots$. Most of the time we suppress the parts of size zero from the notation, for instance the partitions $(4,2,1,0,0)$ and $(4,2,1)$ are considered to be the same; their size is $7=4+2+1$. When $\x$ has repeated parts, we often use the abbreviation $(b^a)$ for the sequence $(b,b,\cdots,b)$ of length $a$. For instance, $(4,4,4,3,3,3,3,3,2,1)$ would be abbreviated as $(4^3,3^5,2,1)$. We denote by $\mc{P}_n$ the collection of partitions with at most $n$ non-zero parts. It is often convenient to identify a partition $\x$ with the associated Young diagram:
\[\Yvcentermath1 \x = (4,2,1,0,0) \quad\quad\llra\quad\quad \yng(4,2,1)\]
When we refer to a row/column of $\x$, we will mean a row/column of the associated Young diagram. Given a partition $\x$, we can then construct the conjugate partition $\x'$ by transposing the associate Young diagrams: $x'_i$ counts the number of boxes in the $i$-th column of $\x$, e.g. $(4,2,1)'=(3,2,1,1)$. Given a positive integer $c$, we write $\x(c)$ for the partition defined by $\x(c)_i=\min(x_i,c)$: the non-zero columns of $\x(c)$ are precisely the first $c$ columns of $\x$.

Cauchy's identity yields the decomposition of $S$ into irreducible $\GL$-representations
\begin{equation}\label{eq:Cauchy}
S = \bigoplus_{\x\in\P_n}S_{\x}\bb{C}^m\oo S_{\x}\bb{C}^n,
\end{equation}
where $S_{\x}$ denotes the Schur functor associated to the partition $\x$. By abusing notation we will write
\[S_{\x}=S_{\x}\bb{C}^m\oo S_{\x}\bb{C}^n,\rm{ so that }S=\bigoplus_{\x\in\P_n} S_{\x}.\]
Very concretely, we can realize $S_{\x}$ as the linear span of the $\GL$-orbit of a highest weight vector as follows. Thinking of $S$ as $\bb{C}[x_{ij}]$, $1\leq i\leq m$, $1\leq j\leq n$, we define for each $l=1,\cdots,n$ the polynomial 
\begin{equation}\label{eq:detl}
\det_l = \det(x_{ij})_{1\leq i,j\leq l}.
\end{equation}
For $\x\in\P_n$ we define
\begin{equation}\label{eq:detx}
 \det_{\x} = \prod_{i=1}^{x_1} \det_{x'_i}.
\end{equation}
We have that $S_{\x}$ is the linear span of the orbit $\GL\cdot\det_{\x}$.

Given any partitions $\x,\y\in\P_n$, we write $\x\leq\y$ if $x_i\leq y_i$ for all $i$. We say that $\x$ and $\y$ are \defi{incomparable} if neither $\x\leq\y$ nor $\y\leq x$, as it is for instance the case when $\x=(2,2)$ and $\y=(3,1)$. We define $I_{\x}$ to be the ideal in $S$ generated by the component $S_{\x}$ in the decomposition (\ref{eq:Cauchy}). It is shown in \cite{deconcini-eisenbud-procesi} that 
\begin{equation}\label{eq:dec-Ix}
 I_{\x} = \bigoplus_{\y\geq\x} S_{\y}
\end{equation}
or equivalently, for $\x,\y\in\P_n$ one has
\begin{equation}\label{eq:inclIxIy}
\x\leq\y \Longleftrightarrow I_{\y}\subseteq I_{\x}.
\end{equation}
Since any $\GL$-invariant ideal $I\subset S$ is minimally generated by a finite number of irreducible $\GL$-representations, we get that it is of the form 
\begin{equation}\label{eq:defIX}
 I_{\mc{X}}=\sum_{\ul{x}\in\mc{X}} I_{\ul{x}},\rm{ for }\X\subseteq\P_n,
\end{equation}
where $\X$ may be further assumed to consist of incomparable partitions. If we write $\sup(\x,\y)$ for the partition defined via
\begin{equation}\label{eq:supxy}
 \sup(\x,\y)_i = \max(x_i,y_i)
\end{equation}
then it follows from (\ref{eq:dec-Ix}) that
\begin{equation}\label{eq:IXcapIY}
 I_{\X} \cap I_{\Y} = \sum_{\x\in\X,\y\in\Y} I_{\sup(\x,\y)}.
\end{equation}

We now recall the definition of the subquotients $J_{\z,l}$ from \cite[Sec.~2B]{raicu-weyman}, which will play an essential role throughout the paper. For $l=0,\cdots,n$ and $\z\in\P_n$, we consider the collection of partitions obtained from $\z$ by adding a single box to its Young diagram in row $(l+1)$ or higher:
\begin{equation}\label{eq:defSucc}
 \mf{succ}(\z,l) = \{\x\in\P_n:\x\geq\z\rm{ and }x_i>z_i\rm{ for some }i>l\}.
\end{equation}
We define (see \cite[(2-6)]{raicu-weyman})
\begin{equation}\label{eq:defJzl}
 J_{\z,l} = I_{\z}/I_{\mf{succ}(\z,l)}.
\end{equation}
Even though this definition makes sense in general, we will only consider $J_{\z,l}$ when the partition $\z$ has the property that $z_1=z_2=\cdots=z_{l+1}$. When $\z=(4^3,3,1)$, $l=2$ and $n\geq 6$, we get that
\begin{equation}\label{eq:anX-Isucc}
I_{\mf{succ}(\z,l)} = I_{\X},\rm{ where }\X=\{(5^3,3,1),(4^4,1),(4^3,3,2),(4^3,3,1,1)\}.
\end{equation}

To every $(\z,l)$ with $z_1=\cdots=z_{l+1}$, we associate the collection of rectangular partitions
\begin{equation}\label{eq:def-Yzl}
 \Y_{\z,l} = \{ ((z_1+1)^{l+1}) \} \cup \{ ((z_i+1)^i) : i>l+1\rm{ and }z_{i-1}>z_i\}.
\end{equation}
In the example above where $\z=(4^3,3,1)$, $l=2$ and $n\geq 6$ we get
\[\Y_{\z,l} = \{(5^3),(4^4),(2^5),(1^6)\}.\]
If $\X$ is as in (\ref{eq:anX-Isucc}) then it follows from (\ref{eq:IXcapIY}) that
\[I_{\z}\cap I_{\Y_{\z,l}} = I_{\X} = I_{\mf{succ}(\z,l)}.\]
$I_{\Y_{\z,l}}$ is the maximal $\GL$-invariant ideal which satisfies this equality, as we explain next.

\begin{lemma}\label{lem:Jzl=Iz-cap-IY}
 If $\z$ is a partition with $z_1=\cdots=z_{l+1}$ then $I_{\z}\cap I_{\Y_{\z,l}} = I_{\mf{succ}(\z,l)}$. Moreover, if $I_{\Y}$ is a $\GL$-invariant ideal with the property that $I_{\z}\cap I_{\Y} \subseteq I_{\mf{succ}(\z,l)}$ then $I_{\Y}\subseteq I_{\Y_{\z,l}}$.
\end{lemma}

\begin{proof} Let us first prove the equality $I_{\z}\cap I_{\Y_{\z,l}} = I_{\mf{succ}(\z,l)}$.

``$\subseteq$": if $S_{\x}\subset I_{\z}\cap I_{\Y_{\z,l}}$ then it follows from (\ref{eq:dec-Ix}) that $\x\geq\z$ and either $x_{l+1}\geq z_1+1>z_{l+1}$, or $x_i\geq z_i+1>z_i$ for some $i>l+1$. This means that $\x\in\mf{succ}(\z,l)$, i.e. $S_{\x}\subset I_{\mf{succ}(\z,l)}$.

``$\supseteq$": if $S_{\x}\subset I_{\mf{succ}(\z,l)}$ then $\x\in\mf{succ}(\z,l)$, so $\x\geq\z$, i.e. $S_{\x}\subset I_{\z}$. Moreover, we have that $x_i>z_i$ for some $i>l$. Let us consider the minimal $i$ for which this inequality holds. If $i=l+1$ then $x_{l+1}>z_{l+1}=z_1$, which implies $\x\geq(z_1+1)^{l+1}$ and thus $S_{\x}\subset I_{\Y_{\z,l}}$. If $i>l+1$ then by the minimality of $i$ we have $x_{i-1}=z_{i-1}$ and $x_i>z_i$. This shows that $(z_i+1)^i\in\Y_{\z,l}$ and $\x\geq(z_i+1)^i$, thus $S_{\x}\subset I_{\Y_{\z,l}}$.

To prove the final statement of the Lemma, assume that $I_{\z}\cap I_{\Y} \subseteq I_{\mf{succ}(\z,l)}$ and that $\x$ is such that $S_{\x}\subset I_{\Y}$. It follows that $I_{\sup(\z,\x)}=I_{\z}\cap I_{\x}\subseteq I_{\mf{succ}(\z,l)}$, and thus $\sup(\z,\x)\in\mf{succ}(\z,l)$. This means that $x_i>z_i$ for some $i>l$, so choosing the minimal such $i$ we can argue as in the proof of ``$\supseteq$" to conclude that $S_{\x}\subset I_{\Y_{\z,l}}$.
\end{proof}

\begin{corollary}\label{cor:Jzl-sub-IY}
 There exists a $\GL$-equivariant inclusion of $S$-modules $J_{\z,l}\subseteq S/I_{\Y}$ if and only if $I_{\Y_{\z,l}}\supseteq I_{\Y}\supseteq I_{\mf{succ}(\z,l)}$. Moreover, such an inclusion is uniquely defined up to a scalar.
\end{corollary}

\begin{proof}
 If $I_{\Y_{\z,l}}\supseteq I_{\Y}\supseteq I_{\mf{succ}(\z,l)}$ then
 \[I_{\mf{succ}(\z,l)}\overset{\rm{Lemma }\ref{lem:Jzl=Iz-cap-IY}}{=}I_{\z}\cap I_{\Y_{\z,l}}\supseteq I_{\z}\cap I_{\Y}\supseteq I_{\z}\cap I_{\mf{succ}(\z,l)}\overset{(\ref{eq:dec-Ix})}{=}I_{\mf{succ}(\z,l)}.\]
 Since the kernel of the composition $I_{\z}\subseteq S \onto S/I_{\Y}$ is equal to $I_{\z}\cap I_{\Y}$, we obtain an inclusion
 \[J_{\z,l}=\frac{I_{\z}}{I_{\mf{succ}(\z,l)}}=\frac{I_{\z}}{I_{\z}\cap I_{\Y}}\subseteq S/I_{\Y},\]
 which is clearly $\GL$-equivariant.
 
 For the reverse implication, as well as for the uniqueness of the inclusion, note that $J_{\z,l}$ is a quotient of $I_{\z}$, so it is generated as an $S$-module by its $S_{\z}$ component. Since $S/I_{\Y}$ has a multiplicity-free decomposition into irreducible $\GL$-representations, there exist, up to scalar, at most one $\GL$-equivariant map from $S_{\z}$ to $S/I_{\Y}$. Such a map induces (up to scalar) the composition $I_{\z}\subseteq S\onto S/I_{\Y}$, which descends to an $S$-module inclusion map $J_{\z,l}\subseteq S/I_{\Y}$ if and only if $I_{\z}\cap I_{\Y}=I_{\mf{succ}(\z,l)}$. Clearly this forces $I_{\Y}\supseteq I_{\mf{succ}(\z,l)}$, and by the last part of Lemma~\ref{lem:Jzl=Iz-cap-IY} it forces $I_{\Y}\subseteq I_{\Y_{\z,l}}$, as desired.
\end{proof}

Recall the definition of the \defi{saturation} of an ideal $I$ with respect to $J$,
\begin{equation}\label{eq:def-saturation}
I:J^{\infty} = \{f\in S: f\cdot J^d \subseteq I\textrm{ for }d\gg 0\}.
\end{equation}
When $I=I_{\X}$ and $J=I_p$ the saturation can be described concretely as follows. For $\X\subset\P_n$ we define 
\begin{equation}\label{eq:p-saturation}
 \X^{:p} = \{\x(c): \x\in\X,c\in\bb{Z}_{\geq 0}, x'_c> p\rm{ if }c>0,\rm{ and }x'_{c+1}\leq p\}
\end{equation}
In terms of Young diagrams, we can think of $\X^{:p}$ as being obtained from $\X$ by removing from each $\x\in\X$ the columns of size $\leq p$.

\begin{lemma}\label{lem:saturation}
 For every $\X\subset\P_n$, the saturation of $I_{\X}$ with respect to $I_p$ can be described by
\[I_{\X}:I_p^{\infty} = I_{\X^{:p}}.\]
\end{lemma}

\begin{proof} We first show that $I_{\X}:I_p^{\infty} \supseteq I_{\X^{:p}}$. Suppose that $\x\in\X$ and $c\in\bb{Z}_{\geq 0}$ are such that $x'_c > p$ (if $c>0$) and $x'_{c+1}\leq p$. We let $\y=\x(c)$ and prove that for $d\gg 0$ we have $I_{\y}\cdot I_p^d\subseteq I_{\x}\subseteq I_{\X}$ and therefore $I_{\y}\subseteq I_{\X}:I_p^{\infty}$. Using Pieri's rule, we have that for every partition $\z\in\P_n$
 \begin{equation}\label{eq:Pieri}
 I_{\z} \cdot I_p \subseteq \sum_{\t\in\P_n,\ \t/\z = (1^p)} I_{\t},
 \end{equation}
 where the notation $\t/\z=(1^p)$ means that $|\t|=|\z|+p$ and $0\leq t_i-z_i\leq 1$ for all $i=1,\cdots,n$. Applying this iteratively starting with $I_{\y}$, we obtain
 \[I_{\y} \cdot I_p^d \subseteq \sum_{\substack{\t\in\P_n,\ |\t| = |\y|+pd\\ 0\leq t_i-y_i\leq d}} I_{\t}.\]
We claim that if $d\gg 0$ then every $\t$ appearing in the above equation has the property that $\t\geq \x$ and thus $I_{\t}\subseteq I_{\x}$. To prove the claim, assume that there exists a partition $\t\in\P_n$ with $|\t|=|\y|+pd$, $0\leq t_i-y_i\leq d$ and $\t\not\geq\x$. Since $\t\geq\y=\x(c)$, we must have $t_i' < x_i'$ for some $i>c$. Since $x_i'=0$ for $i>x_1$ we get $i\leq x_1$, and since $x'_i\leq x'_{c+1}\leq p$ we get $t_i'\leq p-1$. In any case, we have that for any such $\t$, $t_i' \leq p-1$ for $i\geq x_1+1$. Moreover, $t_i'=0$ for $i>c+d$ because $t_1\leq y_1+d = c+d$. Since $\t\in\P_n$, $t_i'\leq n$ for $i\leq x_1$ and therefore
\[|\y|+p\cdot d = |\t| = |\t'| \leq n\cdot x_1 + (p-1)\cdot(c+d-x_1) \Longrightarrow d\leq n\cdot x_1+(p-1)\cdot(c-x_1) - |\y|.\]
Since the above inequality fails for $d\gg 0$, the desired conclusion follows.

In order to prove that $I_{\X}:I_p^{\infty} \subseteq I_{\X^{:p}}$ we first show that
\begin{equation}\label{eq:Iz+dp}
I_{\z}\cdot I_p^d \supseteq I_{\z+(d^p)}
\end{equation}
To see this, note that using notation (\ref{eq:detl}) and (\ref{eq:detx}) we have
\[\det_{\z+(d^p)} = \det_{\z} \cdot (\det_p)^d \in I_{\z}\cdot I_p^d.\]
Since $I_{\z}$ is generated by the $\GL$-orbit of $\det_{\z+(d^p)}$, (\ref{eq:Iz+dp}) follows.

Let $\y$ be such that $I_{\y}\subseteq I_{\X}:I_p^{\infty}$ and consider a positive integer $d$ such that $I_{\y}\cdot I_p^d \subseteq I_{\X}$. It follows from (\ref{eq:Iz+dp}) that $I_{\y+(d^p)}\subseteq I_{\X}$ and therefore by (\ref{eq:defIX}) and (\ref{eq:inclIxIy}) we can find $\x\in\X$ with $\y+(d^p)\geq \x$. Let $c\in\bb{Z}_{\geq 0}$ be such that $x'_c>p+1$ (if $c>0$) and $x'_{c+1}\leq p$. We claim that $\y\geq\x(c)$, which combined with $\x(c)\in\X^{:p}$ implies $I_{\y}\subseteq I_{\X^{:p}}$, as desired. To prove the claim, we may assume that $c>0$ (otherwise $\y\geq\x(0)=\ul{0}$). Since $x'_c>p+1$ it follows that $x_i\geq c$ for $i\leq p+1$. Since $\y+(d^p)\geq \x$, it follows that $y_i\geq x_i$ for all $i>p$. In particular $y_{p+1}\geq x_{p+1}\geq c$ and therefore $y_i\geq c$ for $i\leq p$. Putting together all these inequalities yields $\y\geq \x(c)$, as desired.
\end{proof}

The following result can be proved using the more geometric description of $J_{\z,l}$ in \cite[Lemma~3.2]{raicu-weyman}. We include here a short algebraic proof for sake of completeness.

\begin{corollary}\label{cor:Ann-Jzl}
 The annihilator of $J_{\z,l}$ is given by
 \[\Ann(J_{\z,l}) = I_{l+1}.\]
 Equivalently, the scheme theoretic support of $J_{\z,l}$ consists of the (reduced) variety of matrices of rank $\leq l$.
\end{corollary}

\begin{proof} 
 Since $J_{\z,l}$ is $\GL$-equivariant, its set-theoretic support is defined by an ideal of minors: $\sqrt{\Ann(J_{\z,l})} = I_p$ for some $p$. We show that $p>l$ and that $I_{l+1}\subseteq \Ann(J_{\z,l})$, from which $\Ann(J_{\z,l}) = I_{l+1}$ follows.
 
 To see that $I_{l+1}\subseteq \Ann(J_{\z,l})$ we use (\ref{eq:Pieri}): we need to show that $I_{l+1}\cdot I_{\z} \subseteq I_{\mf{succ}(\z,l)}$, so it is enough to check that if $\t/\z=(1^{l+1})$ then $\t\in\mf{succ}(\z,l)$. Since $|\t|-|\z|=l+1$ and $0\leq t_i-z_i\leq 1$ for all $i$, there exists $i>l$ such that $t_i>z_i$, so $\t\in\mf{succ}(\z,l)$ by (\ref{eq:defSucc}).
 
 Assume now that $\sqrt{\Ann(J_{\z,l})} = I_p$, so that $I_p^d\subseteq\Ann(J_{\z,l})$ for some $d\gg 0$, or equivalently $I_p^d\cdot I_{\z} \subseteq I_{\mf{succ}(\z,l)}$. Using (\ref{eq:Iz+dp}), it follows that $I_{\z+(d^p)} \subseteq I_{\mf{succ}(\z,l)}$ and therefore $\z+(d^p) \in \mf{succ}(\z,l)$ which implies $p>l$.
\end{proof}

\subsection{$\Ext$ modules for the subquotients $J_{\z,l}$ and regularity}\label{subsec:ExtJzl}

In this section we recall the calculation of $\Ext^{\bullet}_S(J_{\z,l},S)$ from \cite[Thm.~3.3]{raicu-weyman}, and use it to give a formula for the regularity of $J_{\z,l}$.

\begin{theorem}\label{thm:ExtJzl}
 Fix an integer $0\leq l\leq n$ and assume that $\z\in\P_n$ is a partition with $z_1=z_2=\cdots=z_l$. For $0\leq s\leq t_1\leq\cdots\leq t_{n-l}\leq l$ we consider the set $W(\z,l;\t,s)$ of dominant weights $\ll\in\bb{Z}^n_{\dom}$ satisfying
 \begin{equation}\label{eq:lam-in-W}
 \begin{cases}
 \ll_n \geq l - z_l - m, & \\
 \ll_{t_i+i} = t_i - z_{n+1-i} - m \quad\rm{for}\quad i=1,\cdots,n-l, & \\
 \ll_s \geq s-n \quad\rm{and}\quad \ll_{s+1} \leq s-m. & \\
 \end{cases}
 \end{equation}
 Letting $\ll(s) = (\ll_1,\cdots,\ll_s,(s-n)^{m-n},\ll_{s+1}+(m-n),\cdots,\ll_n+(m-n))\in\bb{Z}^m_{\dom}$ as in \cite[(1-2)]{raicu-weyman}, we~have
 \begin{equation}\label{eq:Extj}
 \Ext^j_S(J_{\z,l},S) = \bigoplus_{\substack{0\leq s\leq t_1\leq\cdots\leq t_{n-l}\leq l \\ m\cdot n - l^2 - s\cdot(m-n) - 2\cdot\left(\sum_{i=1}^{n-l} t_i\right)=j \\ \ll\in W(\z,l;\t,s)}} S_{\ll(s)}\bb{C}^m \oo S_{\ll}\bb{C}^n,
 \end{equation}
 where $S_{\ll(s)}\bb{C}^m \oo S_{\ll}\bb{C}^n$ appears in degree $|\ll|$. If in addition $z_{l+1}=z_l$, then every weight $\ll$ in (\ref{eq:Extj}) satisfies
 \begin{equation}\label{eq:lln=zl-m}
  \ll_n = l - z_l - m.
 \end{equation}
\end{theorem}

\begin{proof}
 We only need to verify the assertion (\ref{eq:lln=zl-m}) since everything else is part of \cite[Thm.~3.3]{raicu-weyman}. Consider a weight $\ll\in W(\z,l;\t,s)$ for some $0\leq s\leq t_1\leq\cdots\leq t_{n-l}\leq l$. Letting $i=n-l$ in (\ref{eq:lam-in-W}) we obtain 
 \[t_{n-l}-z_{l+1}-m=\ll_{t_{n-l}+n-l}\geq\ll_n\geq l-z_l-m.\]
 Since $t_{n-l}\leq l$ and $z_{l+1}=z_l$, we must have an equality throughout, hence (\ref{eq:lln=zl-m}) must hold.
\end{proof}

Using Theorem~\ref{thm:ExtJzl} and the fact that the regularity of an $S$-module $M$ can be computed via (\ref{eq:regFromExt}), we can determine $\reg(J_{\z,l})$. For $0\leq l\leq n$ and $\z\in\P_n$, we define
\begin{equation}\label{eq:defTlz}
\T_l(\z) = \{\t=(t_1,\cdots,t_{n-l},t_{n-l+1}=l)\in\bb{Z}^{n-l+1}_{\geq 0}:0\leq t_{i+1}-t_i\leq z_{n-i}-z_{n+1-i}\rm{ for }i=1,\cdots,n-l\}.
\end{equation}
For $\t\in\T_l(\z)$ we write
\begin{equation}\label{eq:def-flzt}
 f_l(\z,\t) = \sum_{i=1}^{n-l} t_i\cdot (z_{n-i} - z_{n+1-i} - t_{i+1} + t_i),
\end{equation}
and note that $f_l(\z,\t)\geq 0$. As usual, we write $|\t|=t_1 + \cdots + t_{n-l+1}$ for the size of $\t$.

\begin{theorem}\label{thm:regJzl}
For $0\leq l\leq n-1$ and $\z\in\P_n$ satisfying $z_1=\cdots=z_l=z_{l+1}$, we have using the notation above that
 \begin{equation}\label{eq:regJzl}
 \reg(J_{\z,l}) = \max_{\t\in\T_l(\z)} \left(|\z| + |\t| - l - f_l(\z,\t)\right).
 \end{equation}
\end{theorem}

\begin{example}\label{ex:regJz1}
 Let $n\geq 2$, $\z=(d,d)$ for some $0\leq d\leq n-2$, and $l=1$. We claim that
\[\reg(J_{\z,l})=d+n-1.\]
Consider first the case when $d=0$. We have that $\T_l(\z)$ consists of a single element, namely $\t=(1^n)$. Since $|\t|=n$, $|\z|=0$, and $f_l(\z,\t)=0$, it follows that $\reg(J_{\z,l})=n-1$, as desired. Observe that $J_{\z,l}=S/I_2$, where $I_2$ is the ideal of $2\times 2$ minors of the generic $m\times n$ matrix, which defines the Segre embedding of $\bb{P}^{m-1}\times\bb{P}^{n-1}$ inside $\bb{P}^{mn-1}$. This is well-known to have Castelnuovo--Mumford regularity $n-1$.

Assume now that $d\geq 1$. We have $|\z|=2d$, and
\[\T_l(\z) = \{\t^1=(0^{n-2},1^2),\t^2=(1^n)\}\]
\[\ |\t^1|=2,\ |\t^2|=n,\ f_l(\z,\t^1) = 0,\ f_l(\z,\t^2) = d.\]
Using Theorem~\ref{thm:regJzl} we obtain
\[\reg(J_{\z,l}) = \max\{2d+2 - 1 - 0, 2d + n - 1 - d\}=d+n-1.\]
\end{example}

\begin{proof}[Proof of Thm.~\ref{thm:regJzl}]
Using (\ref{eq:Extj}) and (\ref{eq:regFromExt}) we get that
\begin{equation}\label{eq:1st-regJzl}
\reg(J_{\z,l}) = \max_{\ll,s,\t}\left\{-|\ll|-m\cdot n + l^2 + s\cdot(m-n) + 2\cdot\left(\sum_{i=1}^{n-l} t_i\right)\right\},
\end{equation}
where $\ll,s,\t=(t_1,t_2,\cdots,t_{n-l},t_{n-l+1}=l)$ satisfy the conditions in Theorem~\ref{thm:ExtJzl}. Note that the condition $s\leq t_1$ together with the fact that $\ll$ is dominant implies
\begin{equation}\label{eq:lowerbound-s}
s-m\geq\ll_{s+1}\geq\ll_{t_1+1} = t_1 - z_n - m,\rm{ i.e. }s\geq t_1-z_n.
\end{equation}
Likewise, for $i=1,\cdots,n-l-1$,
\begin{equation}\label{eq:ineq-difft-diffz}
 t_i - z_{n+1-i} - m = \ll_{t_i+i} \geq \ll_{t_{i+1}+i+1} = t_{i+1} - z_{n-i} - m,\rm{ i.e. }t_{i+1} - t_i \leq z_{n-i} - z_{n+1-i}.
\end{equation}
Finally,
\begin{equation}\label{eq:ineq-lastdifftz}
 t_{n-l} - z_{l+1} - m = \ll_{t_{n-l}+n-l} \geq \ll_n \geq l - z_l - m,\rm{ i.e. }t_{n-l+1} - t_{n-l} = l - t_{n-l} \leq z_l - z_{l+1}.
\end{equation}
Note that conditions (\ref{eq:ineq-difft-diffz}) and (\ref{eq:ineq-lastdifftz}) are equivalent to the condition that $\t\in\T_l(\z)$. Moreover, if the conditions (\ref{eq:lowerbound-s})--(\ref{eq:ineq-lastdifftz}) are satisfied, then there exists at least one dominant weight $\ll$ satisfying the conditions in Theorem~\ref{thm:ExtJzl}, and the minimal such $\ll$ (i.e. the one for which $-|\ll|$ is maximal) is given by
\[\ll_1=\cdots=\ll_s=s-n,\]
\[\ll_{s+1}=\cdots=\ll_{t_1+1} = t_1 - z_n - m,\]
\[\ll_{t_i + i + 1} = \cdots = \ll_{t_{i+1} + i + 1} = t_{i+1} - z_{n-i} - m,\rm{ for }i=1,\cdots,n-l-1,\]
\[\ll_{t_{n-l} + n - l + 1} = \cdots = \ll_n = l - z_l - m.\]
Observe that the weight $\ll$ has size
\[|\ll| = s\cdot(s-n) + (t_1-s+1)\cdot(t_1-z_n-m) + \left(\sum_{i=1}^{n-l-1} (t_{i+1}-t_i+1)\cdot(t_{i+1}-z_{n-i}-m)\right) + (l-t_{n-l})\cdot(l-z_l-m).\]
The coefficient of $m$ in $|\ll|$ is given by
\[-\left(t_1-s+1 + \left(\sum_{i=1}^{n-l-1} (t_{i+1}-t_i+1)\right) + (l-t_{n-l})\right) = s-n,\]
so we can rewrite $|\ll|$ as
\[
\begin{aligned}
|\ll| = & (m+s)\cdot(s-n) + (t_1-s+1)\cdot(t_1-z_n) + \left(\sum_{i=1}^{n-l-1} (t_{i+1}-t_i+1)\cdot(t_{i+1}-z_{n-i})\right) + (l-t_{n-l})\cdot(l-z_l) \\
= & (m+s)\cdot(s-n) + (t_1-s)\cdot(t_1-z_n) + \left(\sum_{i=0}^{n-l-1}(t_{i+1}-z_{n-i})\right) + \left(\sum_{i=1}^{n-l} (t_{i+1}-t_i)\cdot(t_{i+1}-z_{n-i})\right)
\end{aligned}
\]
Using the fact that $-(m+s)\cdot(s-n)-m\cdot n + s\cdot(m-n) = -s^2$, we obtain
\[
\begin{aligned}
  & -|\ll|-m\cdot n + l^2 + s\cdot(m-n) + 2\cdot\left(\sum_{i=1}^{n-l} t_i\right) \\
  = & -s^2-(t_1-s)\cdot(t_1-z_n) + l^2 +  \left(\sum_{i=1}^{n-l} t_i\right) + \left(\sum_{i=l+1}^{n} z_i\right) - \left(\sum_{i=1}^{n-l} (t_{i+1}-t_i)\cdot(t_{i+1}-z_{n-i})\right) \\
\end{aligned}
\] 
According to (\ref{eq:1st-regJzl}), in order to obtain $\reg(J_{\z,l})$, we need to maximize the above quantity over all choices of $\t\in\T_l(\z)$ (see (\ref{eq:ineq-difft-diffz})--(\ref{eq:ineq-lastdifftz})) and all $s$ satisfying $\max(0,t_1-z_n)\leq s\leq t_1$ (see (\ref{eq:lowerbound-s})). Fixing $\t\in\T_l(\z)$, we note that
\[-s^2 + s\cdot(t_1-z_n)=s\cdot(t_1-z_n-s) \leq 0\]
when $\max(0,t_1-z_n)\leq s$, and equality is attained when $s=\max(0,t_1-z_n)$. We may thus assume that this is the case, and we get
\begin{equation}\label{eq:2nd-regJzl}
\reg(J_{\z,l}) = \max_{\t\in\T_l(\z)}\left\{-t_1\cdot(t_1-z_n) + l^2 + \left(\sum_{i=1}^{n-l} t_i\right) + \left(\sum_{i=l+1}^{n} z_i\right) - \left(\sum_{i=1}^{n-l} (t_{i+1}-t_i)\cdot(t_{i+1}-z_{n-i})\right) \right\}
\end{equation}
Comparing with (\ref{eq:def-flzt}) and using the fact that $t_{n-l+1}=l$, we get that
\[-t_1\cdot(t_1-z_n) -  \left(\sum_{i=1}^{n-l} (t_{i+1}-t_i)\cdot(t_{i+1}-z_{n-i})\right) = -f_l(\z,\t) - l^2 + l\cdot z_l.\]
Since $z_1=\cdots=z_l$, it follows that  $\left(\sum_{i=l+1}^{n} z_i\right)+l\cdot z_l = |\z|$. Moreover, $\sum_{i=1}^{n-l} t_i=|\t|-t_{n-l+1}=|\t|-l$, hence (\ref{eq:2nd-regJzl}) yields (\ref{eq:regJzl}), as desired.
\end{proof}

\subsection{$\Ext$-split filtrations}\label{subsec:Ext-split-filtrations}

Let $M$ denote a finitely generated graded $S$-module. In this section we will be concerned with analyzing the graded modules $\Ext^j_S(M,S)$ for $j\geq 0$. The reader may assume that $S=\bb{C}[x_{ij}]$ is the ring of polynomial functions on $m\times n$ complex matrices, and that $M$ is a $\GL$-equivariant $S$-module, in which case each $\Ext^j_S(M,S)$ is itself a $\GL$-representation. We fix a finite filtration of $M$ by graded $S$-submodules
\begin{equation}\label{eq:defMbullet}
 M_{\bullet}:\quad 0=M_0\subseteq M_1\subseteq\cdots\subseteq M_r=M.
\end{equation}

\begin{lemma}\label{lem:ABC}
 Let $A\overset{\phi}{\lra} B\overset{\psi}{\lra} C$ be an exact sequence of graded vector spaces (or $\GL$-representations) and assume that $A_d,B_d,C_d$ are finite dimensional in each degree $d$. We have that $B$ embeds into $A\oplus C$ as a graded vector subspace (or $\GL$-subrepresentation). Moreover, we have that $B\simeq A\oplus C$ if and only if the sequence $0\lra A\overset{\phi}{\lra} B\overset{\psi}{\lra} C\lra 0$ is exact.
\end{lemma}

\begin{proof}
 We may assume that $A,B,C$ are concentrated in a single degree, so they are finite dimensional vector spaces (or $\GL$-representations). We get a short exact sequence $0\lra\ker(\psi)\lra B\lra\rm{Im}(\psi)\lra 0$. Since any such sequence splits, we have $B\simeq\ker(\psi)\oplus\rm{Im}(\psi)\simeq\rm{Im}(\phi)\oplus\rm{Im}(\psi)$. By choosing a splitting of the surjection $A\onto\rm{Im}(\phi)$ we can identify $\rm{Im}(\phi)$ with a subspace (subrepresentation) of $A$. This shows that $B$ embeds into $A\oplus C$. Since $A,B,C$ are finite dimensional, we get $B\simeq A\oplus C$ if and only if $\rm{Im}(\phi)\simeq A$ and $\rm{Im}(\psi)\simeq C$ if and only if $0\lra A\overset{\phi}{\lra} B\overset{\psi}{\lra} C\lra 0$ is exact.
\end{proof}

\begin{lemma}\label{lem:boundExt}
 For each $j\geq 0$, the graded vector space 
 \[\Ext^j_S(M,S)\rm{ embeds as a graded subspace of }\bigoplus_{i=0}^{r-1} \Ext^j_S(M_{i+1}/M_i,S).\]
 If $M_{\bullet}$ is a $\GL$-equivariant filtration then the embedding can be chosen to be $\GL$-equivariant.
\end{lemma}

\begin{proof}
 We prove by descending induction on $i$ that $\Ext^j_S(M/M_i,S)$ embeds into $\bigoplus_{k=i}^{r-1} \Ext^j_S(M_{k+1}/M_k,S)$, the case $i=r-1$ being a direct consequence of the equality $M=M_r$. The exact sequence
 \[0\lra M_{i+1}/M_i \lra M/M_i \lra M/M_{i+1} \lra 0\]
 induces for each $j\geq 0$ an exact sequence of graded vector spaces ($\GL$-representations)
 \begin{equation}\label{eq:seq-Mquots}
 \Ext^j_S(M/M_{i+1},S)\lra \Ext^j_S(M/M_i,S)\lra\Ext^j_S(M_{i+1}/M_i,S).
 \end{equation}
 By Lemma~\ref{lem:ABC}, we get that $\Ext^j_S(M/M_i,S)$ embeds into $\Ext^j_S(M/M_{i+1},S)\oplus\Ext^j_S(M_{i+1}/M_i,S)$ and the desired conclusion follows by induction.
\end{proof}

The filtration $M_{\bullet}$ is said to be \defi{$\Ext$-split} if for each $j\geq 0$ we have an isomorphism of graded vector spaces
 \begin{equation}\label{eq:ExtM-iso}
 \Ext^j_S(M,S) \simeq \bigoplus_{i=0}^{r-1} \Ext^j_S(M_{i+1}/M_i,S).
 \end{equation}
 We note that we don't require that (\ref{eq:ExtM-iso}) is an isomorphism of $S$-modules. We also note that (\ref{eq:ExtM-iso}) is equivalent to the fact that the complex
 \begin{equation}\label{eq:ses-Mquots}
 0\lra\Ext^j_S(M/M_{i+1},S)\lra \Ext^j_S(M/M_i,S)\lra\Ext^j_S(M_{i+1}/M_i,S)\lra 0
 \end{equation}
is exact for each $j\geq 0$. If $M_{\bullet}$ is $\Ext$-split, we will refer to the subquotients $M_{i+1}/M_i$ as \defi{$\Ext$-factors} of $M$, noting that they depend on the filtration $M_{\bullet}$ and not just on the module $M$. We will mainly be concerned with the case when $M$ is a $\GL$-equivariant $S$-module, and the filtration (\ref{eq:defMbullet}) is itself $\GL$-equivariant. In this case, the factors $M_{i+1}/M_i$ are $\GL$-representations, as well as all the $\Ext$ modules in (\ref{eq:ExtM-iso}--\ref{eq:ses-Mquots}), while the maps between them are $\GL$-equivariant. 

\section{An $\Ext$-split filtration for $\GL$-equivariant thickenings}\label{sec:Ext-split-thick}

Let $S$ be the ring of polynomial functions on $m\times n$ matrices, and let $I\subseteq S$ denote a $\GL$-invariant ideal. In this section we construct a $\GL$-equivariant $\Ext$-split filtration of $S/I$ (as defined in Section~\ref{subsec:Ext-split-filtrations}), all of whose factors are of the form $J_{\z,l}$ for some partition $\z$ and some non-negative integer~$l$. Using the results from Section~\ref{subsec:ExtJzl}, this will allow us to compute the $\GL$-structure of the modules $\Ext^j_S(S/I,S)$ and in particular to detect the regularity of $I$. To parametrize the factors of the $\Ext$-split filtration, we use the following definition.

\begin{definition}\label{def:ZX}
 For $\mc{X}\subset\P_n$ a finite subset we define $\mc{Z}(\mc{X})$ to be the set consisting of pairs $(\z,l)$ where $\z\in\P_n$ and $l\geq 0$ are such that if we write $c=z_1$ then the following hold:
 \begin{enumerate}
  \item There exists a partition $\ul{x}\in\mc{X}$ such that $\ul{x}(c)\leq\z$ and $x'_{c+1}\leq l+1$.
  \item For every partition $\ul{x}\in\mc{X}$ satisfying (1) we have $x'_{c+1}=l+1$.
 \end{enumerate}
\end{definition}

The main result of this section is then the following.

\begin{theorem}\label{thm:Ext-split-IX}
 Let $\X\subseteq\P_n$ and let $I_{\X}\subseteq S$ denote the associated $\GL$-invariant ideal. There exists a $\GL$-equivariant $\Ext$-split filtration of $S/I_{\X}$, whose factors are the modules $J_{\z,l}$ for $(\z,l)\in\Z(\X)$, and therefore we have for each $j\geq 0$ a $\GL$-equivariant isomorphism of graded vector spaces
 \begin{equation}\label{eq:ExtS/IX}
  \Ext^j_S(S/I_{\X},S) \simeq \bigoplus_{(\ul{z},l)\in\Z(\X)}\Ext^{j}_S(J_{\z,l},S).
 \end{equation}
 In particular, if $I_{\X}\neq S$ we get
 \begin{equation}\label{eq:regIX}
  \reg(I_{\X}) = 1 + \reg(S/I_{\X}) = 1 + \max_{(\ul{z},l)\in\Z(\X)}\reg(J_{\z,l}).
 \end{equation}
 Moreover, if $\X,\Y\subset\P_n$ are such that $I_{\X}\subseteq I_{\Y}$, then the natural surjection $S/I_{\X}\onto S/I_{\Y}$ induces maps $\Ext^j_S(S/I_{\Y},S)\lra\Ext^j_S(S/I_{\X},S)$ for all $j\geq 0$, whose (co)kernels and images can be described via
 \begin{equation}\label{eq:kerExt}
 \ker\left(\Ext^j_S(S/I_{\Y},S)\lra\Ext^j_S(S/I_{\X},S)\right) \simeq \bigoplus_{(\ul{z},l)\in\Z(\Y)\setminus\Z(\X)}\Ext^{j}_S(J_{\z,l},S),
 \end{equation}
 \begin{equation}\label{eq:imExt}
 \operatorname{Im}\left(\Ext^j_S(S/I_{\Y},S)\lra\Ext^j_S(S/I_{\X},S)\right) \simeq \bigoplus_{(\ul{z},l)\in\Z(\Y)\cap\Z(\X)}\Ext^{j}_S(J_{\z,l},S),
 \end{equation}
 \begin{equation}\label{eq:cokerExt}
 \coker\left(\Ext^j_S(S/I_{\Y},S)\lra\Ext^j_S(S/I_{\X},S)\right) \simeq \bigoplus_{(\ul{z},l)\in\Z(\X)\setminus\Z(\Y)}\Ext^{j}_S(J_{\z,l},S).
 \end{equation}
 Finally, recall that the saturation of $I_{\X}$ with respect to $I_p$ is given by $I_{\X^{:p}}$ (see (\ref{eq:p-saturation}), Lemma~\ref{lem:saturation}). We have
 \begin{equation}\label{eq:Z-saturation}
 \Z(\X^{:p}) = \{(\z,l)\in\Z(\X): l\geq p\} \subseteq \Z(\X).
 \end{equation}
 In particular, if we apply (\ref{eq:kerExt}) to the inclusion $I_{\X}\subseteq I_{\X^{:p}}$ we obtain for each $j\geq 0$ injective maps
 \[\Ext^j_S(S/I_{\X^{:p}},S) \lra \Ext^j_S(S/I_{\X},S).\]
\end{theorem}

We begin by recording a couple of useful remarks regarding properties of the set $\Z(\X)$.

\begin{remark}\label{rem:ZdependsonI}
 If the sets $\X,\Y\subset\P_n$ are such that $I_{\X}=I_{\Y}$ then $\Z(\X)=\Z(\Y)$. Using (\ref{eq:inclIxIy}), this is equivalent to the assertion that $\Z(\X)$ depends on the set of minimal partitions in $\X$ rather than on $\X$ itself, which follows directly from Definition~\ref{def:ZX}. 
\end{remark}

\begin{remark}\label{rem:z1--zl+1}
 If $(\z,l)\in\Z(\X)$ then $z_1=\cdots=z_{l+1}$. To see this, let $c=z_1$ and consider any $\x\in\X$ with $\x(c)\leq\z$ and $x'_{c+1}=l+1$. We get that $x_{l+1}\geq c+1$ and therefore $\x(c)_1=\x(c)_2=\cdots=\x(c)_{l+1}=c$. Since $z_1=c$ and $\z\geq\x(c)$, it follows that $z_1=\cdots=z_{l+1}=c$.
\end{remark}

\begin{remark}\label{rem:y-in-ZX}
If $(\z,l)\in\Z(\X)$ then $l$ is uniquely determined by $\z$: writing $c=z_1$, $l$ is the minimum value of $x'_{c+1}-1$, where $\x$ runs over partitions in $\X$ satisfying $x_1>c$ and $\x(c)\leq\z$. It follows that if $(\z,l)\in\Z(\X)$ and $\y\geq\z$, $y_1=z_1$, then either there exists $\x\in\X$ with $x_1\leq c$ and $\x=\x(c)\leq\y$, in which case $I_{\y}\subseteq I_{\X}$, or there exists $u\leq l$ such that $(\y,u)\in\Z(\X)$.
\end{remark}

We begin our proof by verifying (\ref{eq:Z-saturation}), in order to give some more insight into Definition~\ref{def:ZX}.

\begin{proof}[Proof of~(\ref{eq:Z-saturation})]
 Suppose first that $(\z,l)\in\Z(\X^{:p})$ and write $c=z_1$. By Remark~\ref{rem:y-in-ZX}, we have that $l=y'_{c+1}-1$ for some $\y\in\X^{:p}$; since for every $\y\in\X^{:p}$ and for every non-zero $y'_i$ we have $y'_i>p$, it follows that $l\geq p$. It remains to show that $(\z,l)\in\Z(\X)$ and to do so we verify the two conditions in Definition~\ref{def:ZX}. Let $\y\in\X^{:p}$ be such that $\z\geq\y(c)$, $y'_{c+1}=l+1>0$, and write $\y=\x(d)$ for $\x\in\X$ and $d\in\bb{Z}_{\geq 0}$ such that $x'_d>p$, $x'_{d+1}\leq p$. Since $\y$ is obtained from $\x$ by removing all but its first $d$ columns, it must be that $d\geq c+1$. Therefore $\x(c)=\y(c)$ and $x'_{c+1}=y'_{c+1}$, so $\x$ satisfies condition (1) in Definition~\ref{def:ZX}. To check condition (2), assume that $\z\geq\x(c)$ and that $x'_{c+1}\leq l+1$ for some $\x\in\X$. Letting $d\in\bb{Z}_{\geq 0}$ be such that $x'_d>p$, $x'_{d+1}\leq p$, and defining $\y=\x(d)\in\X^{:p}$, we see that $\y(c)\leq\x(c)\leq\z$ and that $y'_{c+1}\leq x'_{c+1}\leq l+1$. Since $(\z,l)\in\Z(\X^{:p})$ we must have $y'_{c+1}=l+1$ and thus $x'_{c+1}=l+1$, which verifies condition (2) as desired.
 
 Suppose now that $(\z,l)\in\Z(\X)$ and that $l\geq p$, let $c=z_1$ and consider any $\x\in\X$ such that $\z\geq\x(c)$ and $x'_{c+1}=l+1>p$. Letting $d\in\bb{Z}_{\geq 0}$ be such that $x'_d>p$ and $x'_{d+1}\leq p$, we see that $d\geq c+1$. We have by (\ref{eq:p-saturation}) that $\y=\x(d)\in\X^{:p}$. Since $c<d$ we have $\y(c)=\x(c)$ and therefore $\z\geq\y(c)$. Moreover, since $\y\leq\x$ we get $y'_{c+1}\leq x'_{c+1}=l+1$ so in order to prove that $(\z,l)\in\Z(\X^{:p})$ it remains to verify condition~(2) in Definition~\ref{def:ZX}. Assume that $\y\in\X^{:p}$ is such that $\z\geq\y(c)$ and $y'_{c+1}\leq l+1$: we need to check that $y'_{c+1}=l+1$. Using (\ref{eq:p-saturation}), we can find $\x\in\X$ such that $\y=\x(d)$ where $d\in\bb{Z}_{\geq 0}$ satisfies $x'_d>p$ and $x'_{d+1}\leq p$. We have two possibilities:
 \begin{itemize}
  \item $c\geq d$: we have $\z\geq\y(c)=\x(d)(c)=\x(d)$. Since $x'_{d+1}\leq p<l+1$ it follows that $x'_i < l+1$ for all $i>d$. By Remark~\ref{rem:z1--zl+1} we know that $z_1=\cdots=z_{l+1}=c$ so $z'_i\geq l+1 \geq x'_i$ for $i=d+1,\cdots,c$, which together with the fact that $\z\geq\x(d)$ shows that $\z\geq\x(c)$. Since $\x\in\X$ and $x'_{c+1}<l+1$, this contradicts the fact that $(\z,l)\in\Z(\X)$.
  
  \item $c<d$: it follows that $\z\geq\y(c)=\x(c)$ and $l+1\geq y'_{c+1}=x'_{c+1}$. Since $(\z,l)\in\Z(\X)$ and $\x\in\X$, we must have $l+1=x'_{c+1}$ and therefore $l+1=y'_{c+1}$, as desired.\qedhere
 \end{itemize}
\end{proof}

To prove the remaining parts of Theorem~\ref{thm:Ext-split-IX}, we proceed in several steps:
\begin{enumerate}
 \item For every $(\z,l)$ where $\z\in\P_n$ is a partition satisfying $z_1=\cdots=z_{l+1}$ (see Remark~\ref{rem:z1--zl+1}), we show that the exact sequence (which exists by Corollary~\ref{cor:Jzl-sub-IY})
 \[0\lra J_{\z,l} \lra \frac{S}{I_{\Y_{\z,l}}} \lra \frac{S}{I_{\z}+I_{\Y_{\z,l}}} \lra 0\]
 induces for each $j\geq 0$ surjective maps at the level of $\Ext$ modules
 \begin{equation}\label{eq:surj-phi^j_zl}
 \phi^j_{\z,l}:\Ext^j_S(S/I_{\Y_{\z,l}},S) \onto \Ext^j_S(J_{\z,l},S).
 \end{equation}
 
 \item The map $\phi^j_{\z,l}$ in (1) is $\GL$-equivariant, so there exists a $\GL$-equivariant splitting for it (in the category of graded vector spaces, not in that of $S$-modules!). We choose a $\GL$-subrepresentation 
 \begin{equation}\label{eq:def-Ejzl}
 \E^j_{\z,l}\subseteq \Ext^j_S(S/I_{\Y_{\z,l}},S)
 \end{equation}
with the property that $\phi^j_{\z,l}$ maps $\E^j_{\z,l}$ isomorphically onto $\Ext^j_S(J_{\z,l},S)$. For every ideal $I_{\Y}\subseteq I_{\Y_{\z,l}}$ we get a natural surjection $S/I_{\Y}\onto S/I_{\Y_{\z,l}}$ and induced maps for each $j\geq 0$
\begin{equation}\label{eq:def-pijzl}
\pi^j_{\z,l,\Y}:\Ext^j_S(S/I_{\Y_{\z,l}},S)\lra\Ext^j_S(S/I_{\Y},S).
\end{equation}
We define
\begin{equation}\label{eq:def-Ejzl(Y)}
\E^j_{\z,l}(\Y) = \pi^j_{\z,l,\Y} (\E^j_{\z,l})\subseteq\Ext^j_S(S/I_{\Y},S)
\end{equation}
and note that this assignment is functorial in the sense that if $I_{\X}\subseteq I_{\Y}\subseteq I_{\Y_{\z,l}}$ and $\pi:S/I_{\X}\onto S/I_{\Y}$ is the corresponding quotient map, then for $j\geq 0$ the induced map
\begin{equation}\label{eq:pi^j-YtoX}
\pi^j_{\Y\to\X}:\Ext^j_S(S/I_{\Y},S) \lra \Ext^j_S(S/I_{\X},S)
\end{equation}
has the property that $\pi^j_{\z,l,\X} = \pi^j_{\Y\to\X}\circ\pi^j_{\z,l,\Y}$, thus
\begin{equation}\label{eq:E^j-YtoX}
\E^j_{\z,l}(\X) = \pi^j_{\Y\to\X}(\E^j_{\z,l}(\Y)).
\end{equation}
We prove that
\begin{equation}\label{eq:E^j_zl-Y}
 \E^j_{\z,l}(\Y)\simeq\begin{cases}
  \E^j_{\z,l} & \rm{when }(\z,l)\in\Z(\Y), \\
  0 & \rm{otherwise}.
 \end{cases}
\end{equation}

\item Finally, we prove that for every ideal $I_{\Y}$ and every $j\geq 0$
\begin{equation}\label{eq:Ext^j=sum-E^j}
\Ext^j_S(S/I_{\Y},S) = \bigoplus_{(\z,l)\in\Z(\Y)} \E^j_{\z,l}(\Y).
\end{equation}
Implicit in this equation is the fact that if $(\z,l)\in\Z(\Y)$ then $I_{\Y}\subseteq I_{\Y_{\z,l}}$, so that $\E^j_{\z,l}(\Y)$ is defined!
\end{enumerate}

\begin{proof}[Proof of Theorem~\ref{thm:Ext-split-IX}]
 The statement about the existence of the $\Ext$-split filtration is proved in Corollary~\ref{cor:filtration-S/IX} and Proposition~\ref{prop:Ext-split-S/IX} below. Equation (\ref{eq:ExtS/IX}) follows from (\ref{eq:E^j_zl-Y}), (\ref{eq:Ext^j=sum-E^j}), and from the fact that $\E^j_{\z,l}$ is isomoprhic to $\Ext^j_S(J_{\z,l},S)$ via the map $\phi^j_{\z,l}$ in (\ref{eq:surj-phi^j_zl}). To prove equations (\ref{eq:kerExt}--\ref{eq:cokerExt}) we observe that by (\ref{eq:E^j-YtoX}), the map $\pi^j_{\Y\to\X}$ is compatible with the decomposition (\ref{eq:Ext^j=sum-E^j}). Moreover, it follows from (\ref{eq:E^j_zl-Y}) that
 \begin{itemize}
 \item when $(\z,l)\in\Z(\X)\cap\Z(\Y)$, the map $\pi^j_{\Y\to\X}$ is in fact an isomorphism between $\E^j_{\z,l}(\Y)$ and $\E^j_{\z,l}(\X)$: this is because it is surjective by (\ref{eq:E^j-YtoX}), the graded vector spaces $\E^j_{\z,l}(\Y)$ and $\E^j_{\z,l}(\X)$ are abstractly isomorphic (they are isomorphic to $\E^j_{\z,l}$), and they are finite dimensional in each degree.
 \item when $(\z,l)\in\Z(\Y)$ but $(\z,l)\not\in\Z(\X)$, $\pi^j_{\Y\to\X}(\E^j_{\z,l}(\Y))=\E^j_{\z,l}(\X)=0$.
 \end{itemize}
 These observations allow us to conclude that
 \[\ker\left(\pi^j_{\Y\to\X}\right) = \bigoplus_{(\ul{z},l)\in\Z(\Y)\setminus\Z(\X)}\E^j_{\z,l}(\Y),\quad \operatorname{Im}\left(\pi^j_{\Y\to\X}\right) = \bigoplus_{(\ul{z},l)\in\Z(\Y)\cap\Z(\X)}\E^j_{\z,l}(\Y),\rm{ and}\]
 \[\coker\left(\pi^j_{\Y\to\X}\right) = \bigoplus_{(\ul{z},l)\in\Z(\X)\setminus\Z(\Y)}\E^j_{\z,l}(\X),\]
 which yield (\ref{eq:kerExt}--\ref{eq:cokerExt}) by applying (\ref{eq:E^j_zl-Y}) one more time and using the isomorphism $\E^j_{\z,l}\simeq\Ext^j_S(J_{\z,l},S)$. 
\end{proof}

We begin by constructing a filtration of $S/I_{\X}$ whose successive quotients are the modules $J_{\z,l}$ for $(\z,l)\in\Z(\X)$. This will be achieved in Corollary~\ref{cor:filtration-S/IX}. For every $c\geq 0$, we associate to $\X$ the set
\begin{equation}\label{def:Xc}
\X(c)=\{\x(c) : \x\in\X\}
\end{equation}
and note that $\X(0)=\{\ul{0}\}$ consists only of the empty partition, and that $\X(c)=\X$ for sufficiently large values of $c$. Moreover, the $\GL$-invariant ideals associated to the sets $\X(c)$ form an eventually constant descending chain
\begin{equation}\label{eq:filtration-IXc}
S=I_{\X(0)}\supseteq I_{\X(1)}\supseteq\cdots\supseteq I_{\X(c)}\supseteq\cdots\supseteq I_{\X}=I_{\X}=\cdots
\end{equation}
We write
\begin{equation}\label{eq:decomposeXc}
 \X(c) = \X(\leq c) \cup \X(>c),\rm{ where}
\end{equation}
\begin{equation}\label{eq:X<=cX>c}
 \X(\leq c) = \{\x(c):\x\in\X\rm{ and }x_1\leq c\} \quad\rm{and}\quad \X(> c) = \{\x(c):\x\in\X\rm{ and }x_1>c\},
\end{equation}
and note that $\X(\leq c)\subseteq\X$ because $\x(c)=\x$ when $x_1\leq c$, and also that $\X(\leq c)\subseteq\X(\leq (c+1))\subseteq\X(c+1)$.

\begin{lemma}\label{lem:filtration-IXc/IXc+1}
 For every $c\geq 0$, the module $M=I_{\X(c)}/I_{\X(c+1)}$ admits a $\GL$-equivariant filtration $0=M_0\subset M_1\subset\cdots$ where the successive quotients $M_{i}/M_{i-1}$ are precisely the modules $J_{\z,l}$ with $(\z,l)\in\Z(\X)$ and $z_1=c$.
\end{lemma}

\begin{proof}
 We first chose an enumeration of the elements $(\z,l)\in\Z(\X)$ with $z_1=c$
\begin{equation}\label{eq:order-zi}
(\z^1,l^1),(\z^2,l^2),\cdots,(\z^r,l^r),
\end{equation} 
having the property that $\z^i\leq\z^j$ can only happen when $i\geq j$. This can be achieved for instance by requiring that $|\z^i|$ is non-increasing as a function of $i$.
 
 We then observe that for every $(\z,l)\in\Z(\X)$ with $z_1=c$ there exists a partition $\x\in\X$ with $\x(c)\leq\z$ and therefore $I_{\z}\subseteq I_{\x(c)}\subseteq I_{\X(c)}$. We define $M_i\subseteq M$ to be the image of the composition
 \[I_{\z^1}+I_{\z^2}+\cdots+I_{\z^i}\subseteq I_{\X(c)} \onto M,\]
 so that
 \begin{equation}\label{eq:Mi-quot}
 M_i = \frac{I_{\z^1}+I_{\z^2}+\cdots+I_{\z^i} + I_{\X(c+1)}}{I_{\X(c+1)}} \subseteq M.
 \end{equation}
It is clear that each $M_i$ is a $\GL$-equivariant submodule of $M$. We have to check that $M_i/M_{i-1}=J_{\z^i,l^i}$ and that $M_r=M$. 

To prove the equality $M_r=M$, it is enough to check that for every minimal element $\z\in\X(c)$ (with respect to $\geq$) we have that either $\z=\z^i$ for some $i$, or $I_{\z}\subseteq I_{\X(c+1)}$. Consider then any minimal $\z\in\X(c)$: if $\z\in\X(\leq c)$ then since $\X(\leq c)\subseteq\X(c+1)$, we get $I_{\z}\subseteq I_{\X(c+1)}$. We may thus assume that $\z\in\X(>c)\setminus\X(\leq c)$ and therefore $\z=\x(c)$ for some $\x\in\X$ with $x_1>c$; we may assume further that $\x$ is chosen to have a minimal value for $x'_{c+1}$ (which is necessarily $>0$ since $x_1>c$). We let $l=x'_{c+1}-1$ and claim that $(\z,l)\in\Z(\X)$, i.e. $\z=\z^i$ for some $i$. Clearly $\z\geq\x(c)$ and $x'_{c+1}\leq l+1$ so if $(\z,l)\notin\Z(\X)$ then it must fail property (2) in Definition~\ref{def:ZX}, i.e. we can find $\y\in\X$ with $\z\geq\y(c)$ and $y'_{c+1}<l+1$. Since $\y(c)\in\X(c)$ and $\z$ was minimal, we must have $\z=\y(c)$. If $y'_{c+1}=0$ then $y_1\leq c$ and therefore $\z=\y(c)\in\X(\leq c)$, a contradiction. It follows that $y'_{c+1}>0$ and thus $y_1>c$, $\y(c)=\z$ and $y'_{c+1}<x'_{c+1}$, contradicting the choice of $\x$.

We next show that $M_i/M_{i-1}=J_{\z^i,l^i}$: to do so, it suffices to verify that the kernel of the natural surjection 
\[I_{\z^i}\onto\frac{I_{\z^1}+I_{\z^2}+\cdots+I_{\z^i} + I_{\X(c+1)}}{I_{\z^1}+I_{\z^2}+\cdots+I_{\z^{i-1}} + I_{\X(c+1)}} \overset{(\ref{eq:Mi-quot})}{=} M_i/M_{i-1}\]
is $I_{\mf{succ}(\z^i,l^i)}$, i.e. it is enough to verify the equality
\begin{equation}\label{eq:Isucc=sumIz}
\begin{aligned}
I_{\mf{succ}(\z^i,l^i)} =\  & I_{\z^i}\cap(I_{\z^1}+I_{\z^2}+\cdots+I_{\z^{i-1}} + I_{\X(c+1)}) \\
\overset{(\ref{eq:IXcapIY})}{=} & I_{\sup(\z^1,\z^i)}+I_{\sup(\z^2,\z^i)}+\cdots+I_{\sup(\z^{i-1},\z^i)} + \sum_{\y\in\X(c+1)}I_{\sup(\y,\z^i)}.
\end{aligned}
\end{equation}
If $\y\in\mf{succ}(\z^i,l^i)$ then we let $\z=\y(c)$ and note that $\z\geq\z^i$ and $z_1=z^i_1=c$. If $I_{\z}\subseteq I_{\X}$ then there exists $\x\in\X$ with $\x\leq \z$, and therefore $\x=\x(c+1)\in\X(c+1)$. It follows that $\y\geq\z\geq\sup(\x,\z^i)$, so $I_{\y}\subseteq I_{\sup(\x,\z^i)}$ is contained in the right hand side of (\ref{eq:Isucc=sumIz}).  If $I_{\z}\not\subseteq I_{\X}$ then since $\z\geq\z^i$ it follows from the last part of Remark~\ref{rem:y-in-ZX} that $(\z,l)\in\Z(X)$ for some~$l$, so $\z=\z^j$ for some $j$. If $\z\neq\z^i$ then $\z^i<\z$, so we must have by the choice of ordering (\ref{eq:order-zi}) that $j<i$, hence $I_{\y}\subseteq I_{\z}=I_{\sup(\z^j,\z^i)}$ is contained in the right hand side of~(\ref{eq:Isucc=sumIz}). If $\z=\z^i$ then since $\y(c)=\z^i$ and $\y\in\mf{succ}(\z^i,l^i)$, we must have $y_k>c$ for some $k>l^i$, thus $y'_{c+1}\geq l^i+1$. Since $(\z^i,l^i)\in\Z(\X)$, there exists $\x\in\X$ with $\x(c)\leq\z\leq\y(c)$ and $x'_{c+1}\leq l^i+1\leq y'_{c+1}$, thus $\y\geq\x(c+1)\in\X(c+1)$. We obtain $\y\geq\sup(\x(c+1),\z^i)$ and therefore $I_{\y}$ is again contained in the right hand side of (\ref{eq:Isucc=sumIz}).

For the reverse inclusion in (\ref{eq:Isucc=sumIz}), let's consider any $\y\in\X(c+1)$, or $\y=\z^j$ for some $j<i$: we have to check that $I_{\sup(\y,\z^i)}\subseteq I_{\mf{succ}(\z^i,l^i)}$, which is equivalent to the fact that $\sup(\y,\z^i)\in\mf{succ}(\z^i,l^i)$. Since $\sup(\y,\z^i)\geq\z^i$, this condition is in turn equivalent to the fact that $y_k>z^i_k$ for some $k>l^i$.

If $y=\z^j$ for some $j<i$ then it follows from the choice of ordering (\ref{eq:order-zi}) that $\sup(\y,\z^i)>\z^i$; moreover, it follows from Remark~\ref{rem:z1--zl+1} that $\sup(\y,\z^i)_k=z^i_k=c$ for $k\leq l^i$, and therefore one must have $y_k>z^i_k$ for some $k>l^i$. Assume now that $\y=\x(c+1)$ for some $\x\in\X$ and that $y_k\leq z^i_k$ for $k>l^i$. It follows that $\x(c)=\y(c)\leq\z^i$ and $x'_{c+1}=y'_{c+1}\leq l^i$ since $y_{l^i+1}\leq z^i_{l^i+1}\leq c$. By condition (2) in Definition~\ref{def:ZX} applied to the pair $(\z^i,l^i)\in\Z(\X)$ and to $\x\in\X$, one must have $y'_{c+1}=x'_{c+1}=l^i+1$, which contradicts $y'_{c+1}\leq l^i$.
\end{proof}

\begin{corollary}\label{cor:filtration-S/IX}
 There is a $\GL$-equivariant filtration of $S/I_{\X}$ whose successive quotients are the modules $J_{\z,l}$ with $(\z,l)\in\Z(\X)$.
\end{corollary}

\begin{proof}
 The desired filtration is obtained by refining (\ref{eq:filtration-IXc}) using the filtrations constructed in Lemma~\ref{lem:filtration-IXc/IXc+1}.
\end{proof}

To verify step (1) of our strategy for proving Theorem~\ref{thm:Ext-split-IX}, we need to check (\ref{eq:surj-phi^j_zl}). This is equivalent to the fact that all the connecting homomorphisms
\begin{equation}\label{eq:conn-hom-Yzl-0}
\Ext^j_S(J_{\z,l},S) \lra \Ext^{j+1}_S(S/(I_{\z}+I_{\Y_{\z,l}}),S)
\end{equation}
are identically zero. Since the homomorphisms (\ref{eq:conn-hom-Yzl-0}) are $\GL$-equivariant, it suffices to prove that all such $\GL$-equivariant maps vanish identically, which is explained as follows. Combining  Corollary~\ref{cor:filtration-S/IX} with Lemma~\ref{lem:boundExt}, we get that $\Ext^{j+1}_S(S/(I_{\z}+I_{\Y_{\z,l}}),S)$ embeds $\GL$-equivariantly into
\[\bigoplus_{(\y,u)\in\Z(\Y_{\z,l}\cup\{\z\})}\Ext^{j+1}_S(J_{\y,u},S)\]
so the desired conclusion follows from the following

\begin{lemma}\label{lem:conn-hom-Yzl-0}
 Let $\Y = \Y_{\z,l} \cup \{\z\}$. If $(\y,u)\in\Z(\Y)$ then there exists no non-zero $\GL$-equivariant maps between $\Ext^j_S(J_{\z,l},S)$ and $\Ext^{j+1}_S(J_{\y,u},S)$.
\end{lemma}

\begin{proof}
 We write $z_1=c$ and begin with the observation that for every partition $\x\in\Y$, $x_1\leq c+1$ (see (\ref{eq:def-Yzl})). Moreover, the non-zero columns of $\x$ have size at least $l+1$ (i.e. $x'_i=0$ or $x'_i\geq l+1$ for all $i$): if $\x=\z$ then it follows from Remark~\ref{rem:z1--zl+1} that its last non-zero column has size $z'_c\geq l+1$; if $\x\in\Y_{\z,l}$ then either $\x=(z_1+1)^{l+1}$ which has all non-zero columns of size $l+1$, or $\x=(z_i+1)^i$ for some $i>l+1$, which has all non-zero columns of size $i>l+1$.

Consider any $(\y,u)\in\Z(\Y)$, and write $y_1=d$. By Definition~\ref{def:ZX}, there exists $\x\in\Y$ such that $\y\geq\x(d)$ and $x'_{d+1} = u+1>0$. It follows from the previous paragraph that $u+1\geq l+1$, i.e. $u\geq l$. Since $x'_{d+1}>0$ we get $x_1\geq d+1$, and since $x_1\leq c+1$ it follows that $d\leq c$. We conclude that
\[l-c\leq u-d,\rm{ with equality if and only if }l=u\rm{ and }c=d.\]
Assume now that there exists a non-zero $\GL$-equivariant map between $\Ext^j_S(J_{\z,l},S)$ and $\Ext^{j+1}_S(J_{\y,u},S)$, which is equivalent to saying that the two share an irreducible $\GL$-subrepresentation $S_{\mu}\bb{C}^m\oo S_{\ll}\bb{C}^n$. By (\ref{eq:lln=zl-m}) we get that
\[l-c-m = \ll_n = u - d - m,\]
so $l=u$ and $c=d$. Furthermore, it follows from (\ref{eq:Extj}) that $\mu=\ll(s)$ for some $s\geq 0$ ($s$ is unique if $m\neq n$), and that $S_{\ll(s)}\bb{C}^m\oo S_{\ll}\bb{C}^n$ appears in $\Ext^j_S(J_{\z,l},S)$ only when
\[j \equiv m\cdot n - l^2 - s\cdot(m-n)\ (\rm{mod }2).\]
By the same reasoning,  $S_{\ll(s)}\bb{C}^m\oo S_{\ll}\bb{C}^n$ appears in $\Ext^{j+1}_S(J_{\y,u},S)$ only when
\[j+1 \equiv m\cdot n - u^2 - s\cdot(m-n)\ (\rm{mod }2).\]
These two congruences can't hold simultaneously, so we obtain the desired conclusion.
\end{proof}

We are now ready to prove that the filtrations in Corollary~\ref{cor:filtration-S/IX} are $\Ext$-split.

\begin{proposition}\label{prop:Ext-split-S/IX}
 If $M_{\bullet}$ is any $\GL$-equivariant filtration of $M = S/I_{\X}$ with successive quotients $J_{\z,l}$ for $(\z,l)\in\Z(\X)$ then $M_{\bullet}$ is $\Ext$-split.
\end{proposition}

\begin{proof}
 Since the 3-term sequences (\ref{eq:seq-Mquots}) are part of a long exact sequence, in order to prove the exactness of (\ref{eq:ses-Mquots}) it suffices to show that for all $i,j\geq 0$ the inclusion $M_{i+1}/M_i\subseteq M/M_i$ induces a surjective map
 \begin{equation}\label{eq:surj-alpha}
 \a:\Ext^j_S(M/M_i,S)\onto\Ext^j_S(M_{i+1}/M_i,S).
 \end{equation}
Let us write $M_{i+1}/M_i = J_{\z,l}$ for some $(\z,l)\in\Z(\X)$. Since $M/M_i$ is a $\GL$-equivariant $S$-module quotient of $S/I_{\X}$, it follows that $M/M_i = S/I_{\Y}$ for some ideal $I_{\Y}\supseteq I_{\X}$. Since the inclusion $M_{i+1}/M_i\subseteq M/M_i$ is $\GL$-equivariant, it follows from Corollary~\ref{cor:Jzl-sub-IY} that $I_{\Y_{\z,l}}\supseteq I_{\Y}$. We obtain natural maps
\[J_{\z,l} \subseteq S/I_{\Y} \onto S/I_{\Y_{\z,l}}\]
which induce at the level of $\Ext$ modules the maps
\[\Ext^j_{S}(S/I_{\Y_{\z,l}},S) \overset{\b}{\lra} \Ext^j_{S}(S/I_{\Y},S) \overset{\a}{\lra} \Ext^j_S(J_{\z,l},S).\]
The composition $\a\circ\b$ coincides with the map $\phi^j_{\z,l}$ in (\ref{eq:surj-phi^j_zl}), so it is surjective. This implies that $\a$ is surjective as well, proving (\ref{eq:surj-alpha}).
\end{proof}

We verify steps (2) and (3) of our strategy for proving Theorem~\ref{thm:Ext-split-IX} in parallel: to do so we have to check (\ref{eq:E^j_zl-Y}) and (\ref{eq:Ext^j=sum-E^j}). We prove half of (\ref{eq:E^j_zl-Y}) together with (\ref{eq:Ext^j=sum-E^j}) in Lemma~\ref{lem:zl-in-ZY} below, and verify the remaining half of (\ref{eq:E^j_zl-Y}) in Lemma~\ref{lem:zl-notin-ZY}.

\begin{lemma}\label{lem:zl-in-ZY}
 With the notation in (\ref{eq:def-Ejzl}) and (\ref{eq:def-Ejzl(Y)}) we have that if $(\z,l)\in\Z(\Y)$ then $\E^j_{\z,l}(\Y)\simeq\E^j_{\z,l}$. Moreover, equation (\ref{eq:Ext^j=sum-E^j}) holds.
\end{lemma}

\begin{proof}
 We choose a $\GL$-equivariant $\Ext$-split filtration $M_{\bullet}$ of $S/I_{\Y}$ as in Proposition~\ref{prop:Ext-split-S/IX}. We write $M/M_i = S/I_{\X^i}$ for some ideal $I_{\X^i}\supseteq I_{\Y}$, and $M_{i+1}/M_i = J_{\z^i,l^i}$ where $(\z^i,l^i)\in\Z(\Y)$. Since $M_{\bullet}$ is $\Ext$-split, it follows from the exactness of (\ref{eq:ses-Mquots}) that the induced maps $\Ext^j_S(M/M_{i+1},S)\lhra\Ext^j_S(M/M_i,S)$ are injective. Note that these are precisely the maps $\pi^j_{\X^{i+1}\to\X^i}$ in (\ref{eq:pi^j-YtoX}). Since $J_{\z^i,l^i}$ is a $\GL$-equivariant submodule of $S/I_{\X^i}$, it follows from Corollary~\ref{cor:Jzl-sub-IY} that $I_{\Y_{\z^i,l^i}}\supseteq I_{\X^i}$, so we get a diagram
 \[
  \xymatrix{
 & S/I_{\Y} = S/I_{\X^0} \ar@{->>}[d] & \\
 J_{\z^i,l^i} \ar@{}[r]|-*[@]{\subseteq} & S/I_{\X^i} \ar@{->>}[r] \ar@{->>}[d] & S/I_{\Y_{\z^i,l^i}} \\
 & S/I_{\X^{i+1}} & \\
 }
 \]
This induces at the level of $\Ext$ modules a diagram
 \[
  \xymatrix{
  & \Ext^j_S(S/I_{\X^{i+1}},S) \ar[d]^{\d} & \\
 \E^j_{\z^i,l^i}\subseteq {\Ext^j_S(S/I_{\Y_{\z^i,l^i}},S)} \ar[r]^(.6){\b} & \Ext^j_S(S/I_{\X^i},S) \ar[r]^{\a} \ar[d]^{\gamma} & \Ext^j_S(J_{\z^i,l^i},S) \\
 & \Ext^j_S(S/I_{\Y},S) &  \\
 }
 \]
We make a couple of observations:
\begin{itemize}
 \item[(a)] The map $\gamma$ is injective since it is the composition of injective maps $\pi^j_{\X^1\to\X^0}\circ\cdots\circ\pi^j_{\X^i\to\X^{i-1}}$.
 \item[(b)] We have that $\b=\pi^j_{\z^i,l^i,\X^i}$ and $\gamma\circ\b = \pi^j_{\z^i,l^i,\Y}$.
 \item[(c)] We have that $\a\circ\b=\phi^j_{\z^i,l^i}$, so it maps $\E^j_{\z^i,l^i}$ isomorphically onto $\Ext^j_S(J_{\z^i,l^i},S)$. It follows that $\b$ maps $\E^j_{\z^i,l^i}$ isomorphically onto $\b(\E^j_{\z^i,l^i})$. Combining this with the injectivity of $\gamma$ we get
 \[\E^j_{\z^i,l^i}\simeq\b(\E^j_{\z^i,l^i})\simeq(\gamma\circ\b)(\E^j_{\z^i,l^i}) = \pi^j_{\z^i,l^i,\Y}(\E^j_{\z^i,l^i}) = \E^j_{\z^i,l^i}(\Y),\]
 which proves the first assertion in our lemma.
 \item[(d)] The map $\d$ coincides with $\pi^j_{\X^{i+1}\to\X^i}$, hence it is injective.
 \item[(e)] We have that $\ker(\a)=\rm{Im}(\d)$, since the sequence 
 \[0\lra \Ext^j_S(S/I_{\X^{i+1}},S) \overset{\d}{\lra} \Ext^j_S(S/I_{\X^i},S) \overset{\a}{\lra} \Ext^j_S(J_{\z^i,l^i},S) \lra 0\]
 is exact: this coincides with the sequence (\ref{eq:ses-Mquots}) and is exact because $M_{\bullet}$ is $\Ext$-split.
 \item[(f)] By (c), $\b(\E^j_{\z^i,l^i})$ maps isomorphically onto the image of $\a$, so combining this with the exact sequence in (e) we get a direct sum decomposition
 \[\Ext^j_S(S/I_{\X^i},S) = \rm{Im}(\d) \oplus \b(\E^j_{\z^i,l^i}) \overset{(b)}{=} \rm{Im}(\d) \oplus \E^j_{\z^i,l^i}(\X^i).\]
\end{itemize}
 
 Applying $\pi^j_{\X^i\to\Y}$ to the decomposition in (f), and using the fact that $\pi^j_{\X^i\to\Y}\circ\d = \pi^j_{\X^{i+1}\to\Y}$ we get
 \[\rm{Im}\left(\pi^j_{\X^i\to\Y}\right) = \rm{Im}\left(\pi^j_{\X^{i+1}\to\Y}\right) \oplus \E^j_{\z^i,l^i}(\Y).\]
 Applying this iteratively for $i=0,1,\cdots$, and using the fact that $I_{\X^0}=I_{\Y}$ and that $\pi^j_{\X^0\to\Y}$ is just the identity on $\Ext^j_S(S/I_{\Y},S)$, we obtain (\ref{eq:Ext^j=sum-E^j}).
 \end{proof}

\begin{lemma}\label{lem:zl-notin-ZY}
 If $I_{\Y}\subseteq I_{\Y_{\z,l}}$ and $(\z,l)\not\in\Z(\Y)$ then $\E^j_{\z,l}(\Y)=0$.
\end{lemma}

\begin{proof}
 We write $z_1=c$, and recall from Remark~\ref{rem:z1--zl+1} that $z_1=\cdots=z_{l+1}=c$. We consider
 \[\X = \{ ((z_i+1)^i) : i>l+1\}.\]
and let $I_{\X}$ denote the associated ideal. We will prove that 
\begin{enumerate}
 \item[(a)] $I_{\Y}\subseteq I_{\X}\subseteq I_{\Y_{\z,l}}$.
 \item[(b)] $\Ext^j_S(J_{\z,l},S)$ and $\Ext^j_S(S/I_{\X},S)$ share no irreducible $\GL$-subrepresentation.
\end{enumerate}
Once we do so, we can use (a) to obtain surjections $S/I_{\Y_{\z,l}}\onto S/I_{\X}\onto S/I_{\Y}$, which yield maps
\[\Ext^j_S(S/I_{\Y_{\z,l}}) \lra \Ext^j_S(S/I_{\X}) \lra \Ext^j_S(S/I_{\Y}).\]
It follows that the map $\pi^j_{\z,l,\Y}$ factors through $\pi^j_{\z,l,\X}$, and therefore $\E^j_{\z,l}(\Y)$ is a homomorphic image of $\E^j_{\z,l}(\X)$, so it suffices to prove that the latter is zero. Since $\E^j_{\z,l}(\X)$ is a $\GL$-subrepresentation of $\Ext^j_S(S/I_{\X},S)$ and a quotient of $\E^j_{\z,l}\simeq\Ext^j_S(J_{\z,l},S)$, it follows from (b) that $\E^j_{\z,l}(\X)=0$, as desired.

To prove assertion (a), we treat the two inclusions separately:

$I_{\Y}\subseteq I_{\X}$: suppose that this isn't the case, and choose $\y\in\Y$ such that $I_{\y}\not\subseteq I_{\X}$. It follows that $y_i\leq z_i$ for all $i>l+1$, which implies $\z\geq\y(c)$. Since $I_{\y}\subseteq I_{\Y}\subseteq I_{\Y_{\z,l}}$, it follows that $\y\geq((z_1+1)^{l+1})$, i.e. $y_{l+1}\geq c+1$. Since we also have $y_{l+2}\leq z_{l+2}\leq c$, it follows that $y'_{c+1}=l+1$. Now since $(\z,l)\notin\Z(\Y)$, it must fail condition (2) in Definition~\ref{def:ZX}, i.e. there must be an element $\x\in\Y$ with $\z\geq\x(c)$ and $x'_{c+1}<l+1$. It follows that $x_i\leq c$ for $i>l$ and since $\z\geq\x(c)$ we get $x_i\leq z_i$ for $i>l$. Since $I_{\x}\subseteq I_{\Y}\subseteq I_{\Y_{\z,l}}$, it follows that $\x\geq((z_1+1)^{l+1})$, which means that $x'_{c+1}\geq l+1$, a contradiction.

$I_{\X}\subseteq I_{\Y_{\z,l}}$: if $\x=((z_i+1)^i)$ for some $i>l+1$ then we have to show that $I_{\x}\subseteq I_{\Y_{\z,l}}$. If $z_i=c$ then $\x\geq((c+1)^i)\geq((z_1+1)^{l+1})\in\Y_{\z,l}$ so $I_{\x}\subseteq I_{\Y_{\z,l}}$. If $z_i<c$, let $j$ be the minimal index for which $z_j=z_i$. We must have $j>l+1$ because $z_{l+1}=c>z_i$, and $z_j<z_{j-1}$ by the minimality of $j$. It follows that $\x\geq((z_j+1)^j)\in\Y_{\z,l}$, so $I_{\x}\subseteq I_{\Y_{\z,l}}$ in this case as well.

To finish our proof, we need to verify assertion (b). Every non-zero column of $\x\in\X$ has size bigger than $l+1$, so it follows from part (2) of Definition~\ref{def:ZX} that if $(\y,u)\in\Z(\X)$ then $u+1>l+1$, i.e. $u>l$. Every $\x\in\X$ satisfies $x_1\leq c+1$, so if $(\y,u)\in\Z(\X)$ then $y_1\leq c$. If $S_{\mu}\bb{C}^m\oo S_{\ll}\bb{C}^n$ occurs as a subrepresentation inside $\Ext^j_S(S/I_{\X},S)$ then there exists $(\y,u)\in\Z(\X)$ such that $S_{\mu}\bb{C}^m\oo S_{\ll}\bb{C}^n$ occurs inside $\Ext^j_S(J_{\y,u},S)$. It follows from (\ref{eq:lln=zl-m}) that $\ll_n = u - y_u - m$. If $S_{\mu}\bb{C}^m\oo S_{\ll}\bb{C}^n$ occurs also inside $\Ext^j_S(J_{\z,l},S)$ then $\ll_n = l - z_l - m$ by the same reasoning. We get
\[u - y_1 = u - y_u = l - z_l = l - c\rm{ hence }u-l=y_1-c,\]
which is in contradiction with the inequalities $y_1\leq c$ and $u>l$.
\end{proof}

\section{An optimization problem}\label{sec:optimization}

In this section we solve an optimization problem that will allow us to determine the regularity of large powers and symbolic powers of generic determinantal ideals. The main result, which will be used throughout Section~\ref{sec:reg-pows}, is Proposition~\ref{prop:Rlpnd-dlarge} below. For $0\leq l<p\leq n$ we consider

\begin{equation}\label{eq:YU-lpnd}
\YU(l,p,n,d) = \left\{ (\y,\u) \in \P_{n-l}\times\P_{n-l} : 
\begin{aligned} 
& |\y|\leq d\cdot(p-l)-1,|\y|-y_1\geq d\cdot(p-1-l) \\  
& u_1 = l, y_i - y_{i+1} \geq u_i - u_{i+1}\rm{ for }1\leq i\leq n-l-1 
\end{aligned}
\right\}
\end{equation}

For $(\y,\u)\in\YU(l,p,n,d)$ we define
\begin{equation}\label{eq:def-glpnd}
 g_{l,p,n,d}(\y,\u) = l\cdot y_1 + |\y| + |\u| - \sum_{i=1}^{n-l-1} u_{i+1}\cdot((y_i-y_{i+1}) - (u_i - u_{i+1})),
\end{equation}
and let
\begin{equation}\label{eq:def-Rlpnd}
R_{l,p,n,d} = \max_{(\y,\u)\in\YU(l,p,n,d)} g_{l,p,n,d}(\y,\u),
\end{equation}
with the convention that $R_{l,p,n,d} = -\infty$ when $\YU(l,p,n,d)$ is the empty set.



\begin{proposition}\label{prop:Rlpnd-dlarge}
 If $0\leq l < p\leq n-1$, or if $p=n$ and $l=p-1$, then
\begin{equation}\label{eq:Rlpnd-dlarge}
R_{l,p,n,d} = p\cdot d - 1 + l\cdot(p-1-l)\rm{ for }d\geq n-1.
\end{equation}
\end{proposition}

\begin{lemma}\label{lem:y1<=d-1}
 Consider integers $0\leq l<p\leq n$. If $(\y,\u)\in\YU(l,p,n,d)$ then 
 \begin{equation}\label{eq:y1<=d-1}
 y_1\leq d-1.
 \end{equation}
\end{lemma}

\begin{proof}
 We note that by (\ref{eq:YU-lpnd})
 \[y_1\leq |\y| - d\cdot(p-1-l) \leq (d\cdot(p-l) - 1) - d\cdot(p-1-l) = d-1.\qedhere\]
 \end{proof}

\begin{lemma}\label{lem:Rlpnd-lowerbound}
 If $0\leq l<p\leq n-1\leq d$ then $R_{l,p,n,d} \geq p\cdot d - 1 + l\cdot(p-1-l)$.
\end{lemma}

\begin{proof}
We consider the partitions $\y,\u\in\P_{n-l}$ defined by
 \[
 \begin{gathered}
 y_1=\cdots=y_{p-l} = d-1,\ y_{p-l+1} = p-l-1,\ y_{p-l+2}=\cdots=y_{n-l}=0, \\
 u_1=\cdots=u_{p-l}=l,\ u_{p-l+1} = \cdots=u_{n-l} = 0,
 \end{gathered}
 \]
 and note that 
 \begin{equation}\label{eq:size-|y|}
 |\y| = d\cdot(p-l) - 1,\ |\y|-y_1 = d\cdot(p-1-l),\rm{ and }|\u|=l\cdot(p-l).
 \end{equation}
 Moreover $u_i - u_{i+1}$ is non-zero only for $i=p-l$, in which case
\[y_{p-l}-y_{p-l+1} = d - 1 - (p - l - 1) = d - p + l \geq (n-1) - p + l \geq l = u_{p-l} - u_{p-l+1},\]
so $y_i-y_{i+1}\geq u_i-u_{i+1}$ for all $i=1,\cdots,n-l-1$. It follows that $(\y,\u) \in \YU(l,p,n,d)$, and we can then use (\ref{eq:def-glpnd}) and (\ref{eq:size-|y|}) to compute
\[g_{l,p,n,d}(\y,\u) = l\cdot(d-1) + (d\cdot(p-l) - 1) + l\cdot(p-l) = p\cdot d - 1 + l\cdot(p-1-l).\]
It follows from (\ref{eq:def-Rlpnd}) that $R_{l,p,n,d}\geq p\cdot d - 1 + l\cdot(p-1-l)$, as desired.
\end{proof}

\begin{lemma}\label{lem:pl=nn-1}
 We have that $R_{n-1,n,n,d} = n\cdot d - 1$ for all $d\geq 1$, and $R_{l,n,n,d}=-\infty$ for $l\leq n-2$. In particular, Proposition~\ref{prop:Rlpnd-dlarge} holds for $p=n$ and $l=n-1$.
\end{lemma}

\begin{proof} We let $p=n$ and $l=n-1$. The partitions $\y,\u$ satisfying the conditions in (\ref{eq:YU-lpnd}) have only one part, $\y=(y_1)$, $\u=(u_1)$, and they satisfy the conditions $y_1\leq d-1$ and $u_1=l=n-1$. Moreover,
\[g_{l,p,n,d}(\y,\u) = l\cdot y_1 + y_1 + u_1 = n\cdot y_1 + n - 1,\]
which is maximized for $y_1 = d-1$. It follows that $R_{n-1,n,n,d} = n\cdot(d-1)+n-1=n\cdot d-1$, as desired.

Assume now that $p=n$ and $l\leq n-2$. If $(\y,\u)\in\YU(l,n,n,d)$ then since $\y$ is non-increasing we obtain
 \[(n-1-l)\cdot y_2 \geq |\y|-y_1\geq d\cdot(p-1-l) = d\cdot(n-1-l),\]
 so $y_2\geq d$, contradicting the chain of inequalities $y_2\leq y_1\leq d-1$ (see Lemma~\ref{lem:y1<=d-1}). Therefore $\YU(l,n,n,d)$ is empty for $l\leq n-2$ and $R_{l,n,n,d}=-\infty$.
\end{proof}

\begin{proof}[Proof of Proposition~\ref{prop:Rlpnd-dlarge}]

By the last part of Lemma~\ref{lem:pl=nn-1}, we may assume that $0\leq l<p\leq n-1$. By Lemma~\ref{lem:Rlpnd-lowerbound}, it suffices to show that
\begin{equation}\label{eq:Rlpnd-upperbound}
 R_{l,p,n,d} \leq p\cdot d - 1 + l\cdot(p-1-l) \rm{ for }d\geq n-1.
\end{equation}
Among the elements $(\y,\u)\in\YU(l,p,n,d)$ for which $g_{l,p,n,d}(\y,\u)=R_{l,p,n,d}$, we consider one for which $\y$ is lexicographically maximal (so in particular the value of $y_1$ is maximal). We claim that
\begin{equation}\label{eq:max-y}
|\y| = d\cdot(p-l) - 1.
\end{equation}
To see this, assume that $|\y|<d\cdot(p-l)-1$ and define a partition $\x\in\P_{n-l}$ by letting $x_1=y_1+1$, and $x_i=y_i$ for $i>1$. We have that $(\x,\u)\in\YU(l,p,n,d)$ and
\[g_{l,p,n,d}(\x,\u) - g_{l,p,n,d}(\y,\u) = l+1 - u_2 = u_1+1-u_2 > 0,\]
contradicting the fact that $g_{l,p,n,d}(\y,\u)=R_{l,p,n,d}\geq g_{l,p,n,d}(\x,\u)$.

To prove (\ref{eq:Rlpnd-upperbound}) we proceed by induction on $n$. We divide our analysis into four cases:

{\bf Case 1: $u_{n-l}=y_{n-l}=0$.} We can think of $\y,\u$ as partitions in $\P_{n-1-l}$ and it follows from (\ref{eq:YU-lpnd}) that $(\y,\u)\in\YU(l,p,n-1,d)$. We get
\[R_{l,p,n,d} = g_{l,p,n,d}(\y,\u) \overset{(\ref{eq:def-glpnd})}{=} g_{l,p,n-1,d}(\y,\u) \leq R_{l,p,n-1,d}.\]
If $p=n-1$ then since $\YU(l,p,n-1,d)$ is non-empty, we must have $l=n-2$ by Lemma~\ref{lem:pl=nn-1} and thus $R_{l,p,n-1,d} = (n-1)\cdot d - 1 =  p\cdot d - 1 + l\cdot(p-1-l)$. If $p<n-1$, then since $d\geq n-1 > (n-1) - 1$, it follows by induction on $n$ that $R_{l,p,n-1,d} = p\cdot d - 1 + l\cdot(p-1-l)$. In both cases we conclude that (\ref{eq:Rlpnd-upperbound}) holds.

{\bf Case 2: $u_{n-l}=0$ and $y_{n-l}>0$.} Suppose first that $y_1<d-1$ and consider the partition $\x\in\P_{n-l}$ defined via $x_1 = y_1 + 1$, $x_i = y_i$ for $i=2,\cdots,n-l-1$, and $x_{n-l} = y_{n-l}-1$. We have that $(\x,\u)\in\YU(l,p,n,d)$ and
\[g_{l,p,n,d}(\x,\u) = g_{l,p,n,d}(\y,\u) + l - (u_2 + u_{n-l}) \overset{(u_{n-l}=0)}{=} g_{l,p,n,d}(\y,\u) + (l - u_2) \overset{(l =u_1\geq u_2)}{\geq} g_{l,p,n,d}(\y,\u).\]
Since $g_{l,p,n,d}(\x,\u)\leq R_{l,p,n,d}=g_{l,p,n,d}(\y,\u)$ by (\ref{eq:def-Rlpnd}), we conclude that $g_{l,p,n,d}(\x,\u)=R_{l,p,n,d}$. Since $x_1>y_1$, this contradicts the maximality of $y_1$.

Assume now that $y_1=d-1$. If there exists $1\leq i < n-l-1$ such that $y_i - y_{i+1} > u_i - u_{i+1}$ then we consider the partition $\x\in\P_{n-l}$ defined via $x_{i+1} = y_{i+1}+1$, $x_{n-l} = y_{n-l}-1$, and $x_j = y_j$ for all $j\neq i+1,n-l$. We have that $(\x,\u)\in\YU(l,p,n,d)$ and
\[g_{l,p,n,d}(\x,\u) = g_{l,p,n,d}(\y,\u) + u_{i+1} - u_{i+2} \geq g_{l,p,n,d}(\y,\u) = R_{l,p,n,d}.\]
It follows that $g_{l,p,n,d}(\x,\u) = R_{l,p,n,d}$, but $\x$ is larger lexicographically than $\y$, so we obtain a contradiction.

The only remaining case is therefore when $y_1 = d-1$ and $y_i - y_{i+1} = u_i - u_{i+1}$ for $i=1,\cdots,n-l-2$. We have
\[g_{l,p,n,d}(\y,\u) \leq l \cdot y_1 + |\y| + |\u| \overset{(\ref{eq:max-y})}{=} l\cdot(d-1) + d\cdot(p-l) - 1 + |\u| = p\cdot d - 1 - l + |\u|.\]
It follows that in order to prove (\ref{eq:Rlpnd-upperbound}) it suffices to verify that $|\u| \leq l\cdot(p-l)$. If $p=n-1$ then since $u_{n-l}=0$ and $u_1=l\geq u_2\geq\cdots\geq u_{n-1-l}$ it follows that 
\[|\u| \leq (n-1-l)\cdot u_1 = (n-1-l)\cdot l = (p-l)\cdot l\]
as desired. We may thus assume that $p+1\leq n-1$. Since $y_1=d-1$ and $u_1=l$, the relations $y_i - y_{i+1} = u_i - u_{i+1}$ for $i=1,\cdots,n-l-2$ imply that
\begin{equation}\label{eq:yi-ui-const}
 d-1-l = y_1 - u_1 = \cdots = y_{n-l-1} - u_{n-l-1}
\end{equation}
and therefore
\[ d\cdot(p-l) - 1 - |\u| \overset{(\ref{eq:max-y})}{=} |\y| - |\u| \geq |\y| - |\u| - y_{n-l} \overset{(\ref{eq:yi-ui-const})}{=} \sum_{i=1}^{n-l-1} (y_i - u_i) = (d-1-l) \cdot (n-1-l) \geq (d-1-l) \cdot (p+1-l),\]
where the last inequality uses $n-1\geq p+1$ and the fact that $d-1-l\geq 0$ which is a consequence of $d\geq n-1\geq p>l$. It follows that
\[ |\u| \leq d\cdot(p-l) - 1 - (d-1-l) \cdot (p+1-l) = l\cdot(p-l) - d + p \leq l\cdot(p-l),\]
where the last inequality follows from the fact that $d\geq n-1\geq p$.

{\bf Case 3: $u_{n-l}=k>0$ and $y_{n-l}\leq (p-1-l) + (d-n+k)$.} Since $0\leq k\leq l$, there exists a bijection
\[
\begin{aligned}
\{(\z,\t)\in\YU(l,p,n,d) : t_{n-l}\geq k\} &\llra \YU(l-k,p-k,n-k,d),\rm{ given by } \\
 (\z,\t) &\llra (\z,\t - (k^{n-l}))
\end{aligned}
\]
Letting $\v = \u - (k^{n-l})$, we get that $(\y,\v)\in\YU(l-k,p-k,n-k,d)$. Moreover, we have
\[g_{l-k,p-k,n-k,d}(\y,\v) = (l-k)\cdot y_1 + |\y| + |\v| - \sum_{i=1}^{n-l-1} v_{i+1}\cdot((y_i-y_{i+1}) - (v_i - v_{i+1})).\]
Since $|\v| = |\u| - k\cdot(n-l)$, $v_i-v_{i+1} = u_i - u_{i+1}$ for all $i=1,\cdots,n-l+1$, and $v_{i+1} = u_{i+1} - k$, we get
\begin{equation}\label{eq:delg-k}
\begin{aligned}
g_{l,p,n,d}(\y,\u) - g_{l-k,p-k,n-k,d}(\y,\v) &= k\cdot y_1 + k\cdot(n-l) - k \cdot\left(\sum_{i=1}^{n-l-1} (y_i-y_{i+1}) - \sum_{i=1}^{n-l-1} (u_i - u_{i+1})\right) \\
 &= k\cdot y_1 + k\cdot(n-l) - k\cdot(y_1-y_{n-l}) + k \cdot(u_1 - u_{n-l}) \\
 &= k\cdot(n-l) + k\cdot y_{n-l} + k \cdot(l - k) \\
 &= k\cdot(n - k + y_{n-l}).
\end{aligned}
\end{equation}
Since $d\geq n-1 > (n-k)-1$, it follows by induction on $n$ that 
\begin{equation}\label{eq:Rl-k-induction}
g_{l-k,p-k,n-k,d}(\y,\v)\leq R_{l-k,p-k,n-k,d} = (p-k)\cdot d - 1 + (l-k)\cdot (p-1-l).
\end{equation}
Our assumption on $y_{n-l}$ can be restated as $n-k+y_{n-l}\leq (p-1-l)+d$, and together with~(\ref{eq:delg-k}) and~(\ref{eq:Rl-k-induction}) it yields
\[
\begin{aligned}
R_{l,p,n,d}=g_{l,p,n,d}(\y,\u) &\leq (p-k)\cdot d - 1 + (l-k)\cdot (p-1-l) + k\cdot(n-k+y_{n-l}) \\
&\leq (p-k)\cdot d - 1 + (l-k)\cdot (p-1-l) + k\cdot((p-1-l)+d) \\
& = p\cdot d - 1 + l\cdot(p-1-l)
\end{aligned}
\]
which proves (\ref{eq:Rlpnd-upperbound}).

{\bf Case 4: $u_{n-l}=k>0$ and $y_{n-l}\geq (p-l) + (d-n+k)$.} Since $u_{n-l}=k$, we can rewrite this as
\[u_{n-l} - y_{n-l} \leq -((p-l) + (d-n)).\]
Since $u_i-u_{i+1} \leq y_i - y_{i+1}$, we get $u_i-y_i\leq u_{i+1}-y_{i+1}$ for all $i=1,\cdots,n-l-1$, and therefore
\[u_i - y_i \leq -((p-l) + (d-n))\rm{ for all }i=1,\cdots,n-l.\]
Adding these inequalities for $i=2,\cdots,n-l$ we obtain
\[|\u|-u_1 - (|\y| - y_1) \leq -(n-l-1)\cdot ((p-l) + (d-n)),\]
or equivalently, since $u_1=l$,
\begin{equation}\label{eq:u-leq}
|\u| \leq |\y| - y_1 + l - (n-l-1)\cdot ((p-l) + (d-n)).
\end{equation}
Moreover, since we have $(y_i - y_{i+1})-(u_i - u_{i+1})\geq 0$ for $i=1,\cdots,n-l-1$, we get from (\ref{eq:def-glpnd}) that
\[
\begin{aligned}
g_{l,p,n,d}(\y,\u) &\leq l\cdot y_1 + |\y| + |\u| \\
&\overset{(\ref{eq:u-leq})}{\leq} (l-1)\cdot y_1 + 2\cdot|\y| + l - (n-l-1)\cdot ((p-l) + (d-n)) \\
&\overset{(\ref{eq:y1<=d-1}),(\ref{eq:max-y})}{\leq} (l-1)\cdot (d-1) + 2\cdot(d\cdot(p-l) - 1) + l - (n-l-1)\cdot ((p-l) + (d-n)) \\
& = (2p-n)\cdot d - 1 - (n-l-1)\cdot (p-l-n)
\end{aligned}
\]
To prove that $g_{l,p,n,d}(\y,\u) \leq p\cdot d - 1 + l\cdot(p-1-l)$, it is then sufficient to verify that
\[(2p-n)\cdot d  - 1 - (n-l-1)\cdot (p-l-n) \leq p\cdot d - 1 + l\cdot(p-1-l),\]
which is equivalent to
\[-l\cdot(p-1-l) + (n-l-1)\cdot(n+l-p) \leq d\cdot(n-p),\]
which is in turn equivalent to $(n-1)\cdot(n-p)\leq d\cdot(n-p)$. This inequality holds for $d\geq n-1$, concluding our proof.
\end{proof}

\section{Regularity of powers of determinantal ideals}\label{sec:reg-pows}

In this section we assume as before that $m\geq n\geq p\geq 1$, that $S=\bb{C}[x_{ij}]$, $1\leq i\leq m$, $1\leq j\leq n$, and that $I_p$ is the ideal generated by the $p\times p$ minors of the generic matrix $(x_{ij})$. We will be interested in understanding the behavior of the regularity of various powers of $I_p$. We recall that the \defi{symbolic powers}~$I_p^{(d)}$ consist of functions vanishing to order at least $d$ along the variety of matrices of rank smaller than $p$ \cite[Thm.~3.14]{eisenbud-CA}. Using notation (\ref{eq:def-saturation}) they can be computed via saturation with respect to $I_{p-1}$
\[I_p^{(d)} = I_p^d : I_{p-1}^{\infty}.\]
We will also be interested in the saturation of the powers $I_p^d$ (with respect to the maximal homogeneous ideal in $S$),
\[(I_p^d)^{sat} = I_p^d : I_1^{\infty}.\]
If we write $\mc{I}_p$ for the ideal sheaf defining the projective variety of matrices of rank smaller than $p$ ($\mc{I}_p$ is the sheafification of the ideal $I_p$), then $(I_p^d)^{sat} = \bigoplus_{r\in\bb{Z}}\Gamma(\bb{P}^{mn-1},\mc{I}_p^d(r))$, and the Castelnuovo-Mumford regularity of $\mc{I}_p^d$ is computed as $\reg(\mc{I}_p^d)=\reg((I_p^d)^{sat})$. Just as with the ideal powers, the regularity of powers of ideal sheaves has been well-studied \cite{cutkosky,cutkosky-ein-lazarsfeld,chardin,niu}, but explicit formulas for the regularity function have been obtained only in very few cases. The main result of this section is as follows (see also Theorem~\ref{thm:2x2} for a more precise result in the case of $2\times 2$ minors).

\begin{theorem}\label{thm:reg-pows}
 Assume that $1\leq p\leq n$. If $p=1$ or $p=n$ then 
 \begin{equation}\label{eq:reg-powers-p=1n}
 \reg(I_p^{d}) = \reg((I_p^d)^{sat}) = \reg(I_p^{(d)})=p\cdot d,\rm{ for all }d\geq 1.
 \end{equation}
 Otherwise, if $1<p<n$ then we have that for $d\geq n-1$
 \begin{equation}\label{eq:reg-large-powers}
 \reg(I_p^{d}) = \reg((I_p^d)^{sat}) = p\cdot d + \max_{0\leq l\leq p-1} l\cdot(p-1-l),\rm{ and }\reg(I_p^{(d)}) = p\cdot d.
 \end{equation}
 If $1<p<n$ and $1\leq d\leq n-2$ then 
 \begin{equation}\label{eq:reg-small-powers}
 \reg(I_p^d)\geq \reg((I_p^d)^{sat})\geq \reg(I_p^{(d)})>p\cdot d.
 \end{equation}
 In particular, if $1<p<n$ then $\reg(I_p^{(d)})$ stabilizes to a linear function in $d$ precisely at $d = n-1$.
\end{theorem}

We note that the expression $l\cdot(p-1-l)$ is maximized for $l=\lfloor (p-1)/2\rfloor$, so the formula for $\reg(I_p^d)$ agrees with the one given in the Theorem on Regularity in the Introduction. Since for an ideal $I$ we have $\reg(I) = \reg(S/I) + 1$, we can use the conclusions of Theorems~\ref{thm:Ext-split-IX} and~\ref{thm:regJzl} to prove our results. It follows from \cite{deconcini-eisenbud-procesi} that $I_p^d=I_{\X_p^d}$ and $I_p^{(d)} = I_{\X_p^{(d)}}$, where
\begin{equation}\label{eq:defXpd}
\mc{X}_p^d = \{\x\in\P_n: |\ul{x}|=p\cdot d,\ x_1\leq d\},\rm{ and }
\end{equation}
\begin{equation}\label{eq:def-X_p^(d)}
\X_p^{(d)} = \{\x\in\P_n: x_1=\cdots=x_p, x_p + x_{p+1} + \cdots + x_n = d\}
\end{equation}

The reader may in fact verify directly that $\X_p^{(d)}$ consists of the minimal elements in $(\X_p^d)^{:(p-1)}$ (see (\ref{eq:p-saturation})), with respect to the ordering $\geq$ on partitions. Therefore, 
\[I_{\X_p^{(d)}} = I_{(\X_p^d)^{:(p-1)}} = I_{\X_p^d} : I_{p-1}^{\infty} = I_p^d : I_{p-1}^{\infty} = I_p^{(d)}. \]
The following is well-known, but we include a short proof for the sake of completeness.

\begin{lemma}\label{lem:Zsymb-vs-Zpow}
 If $p=1$ or $p=n$ then for every $d\geq 1$ we have $I_p^d = (I_p^d)^{sat} = I_p^{(d)}$.
\end{lemma}

\begin{proof} 
Since $I_p^d \subseteq (I_p^d)^{sat} \subseteq I_p^{(d)}$, it is enough to check that $I_p^d=I_p^{(d)}$ when $p=1$ or $p=n$, which in turn is a consequence of the equality $\X_p^{(d)}=\X_p^d$. This follows directly from (\ref{eq:defXpd}) and (\ref{eq:def-X_p^(d)}), since
\[\X_1^{(d)} = \X_1^d = \{\x\in\P_n: |\x|=d\},\rm{ and }\X_n^{(d)} = \X_n^d = \{(d^n)\}.\qedhere\]
\end{proof}

\begin{lemma}\label{lem:Z_p^d}
If we let $\Z_p^{d} = \Z(\X_p^{d})$ (see Definition~\ref{def:ZX}) then
\begin{equation}\label{eq:def-Z_p^d}
\Z_p^{d} = \left\{(\z,l): 
\begin{aligned} 
& 0\leq l\leq p-1,\ \z\in\P_n,\ z_1=\cdots=z_{l+1}\leq d-1, \\  
& |\z|+(d-z_1)\cdot l  + 1\leq p\cdot d \leq |\ul{z}|+(d-z_1)\cdot (l+1)
\end{aligned}
\right\}
\end{equation}
\end{lemma}

\begin{proof} We prove the equality (\ref{eq:def-Z_p^d}) by verifying the double inclusion.

``$\supseteq$": Let $0\leq l\leq p-1$, and $\z\in\P_n$ with $z_1=\cdots=z_{l+1}\leq d-1$. Write $c=z_1$ and assume that
\[|\z|+(d-c)\cdot l  + 1\leq p\cdot d \leq |\ul{z}|+(d-c)\cdot (l+1).\]
We prove that $(\z,l)\in\Z_p^d$ by checking the two properties in Definition~\ref{def:ZX}. To verify (1), we consider any partition $\y\in\P_{l+1}$ such that
\[|\y| = p\cdot d - |\z|,\rm{ and }y_i\leq d-c\rm{ for }i=1,\cdots,l+1.\]
Such a partition exists, since $p\cdot d - |\z| \leq (d-c)\cdot(l+1)$. We define the partition $\x\in\P_n$ via
\[x_i=z_i+y_i\rm{ for }i=1,\cdots,l+1,\rm{ and }x_i=z_i\rm{ for }i=l+2,\cdots,n.\]
We have $|\x| = |\z| + |\y| = p\cdot d$, and $x_1=z_1 + y_1 \leq c + (d-c)=d$, so $\x\in\X_p^d$. Since $x_{l+2}=z_{l+2}\leq z_{l+1}=c$, we get $x'_{c+1}\leq l+1$. We have moreover that $\x(c)=\z$, so property (1) holds.

To check (2), consider a partition $\x\in\X_p^d$ with $\x(c)\leq\z$ and $x'_{c+1}\leq l+1$. We need to verify that $x'_{c+1}=l+1$. Suppose instead that $x'_{c+1}\leq l$, so $x'_i\leq l$ for all $i\geq c+1$. Since $x_1\leq d$, we have $x'_i=0$ for $i>d$. We obtain
\[ p\cdot d = |\x| = |\x(c)| + \sum_{i=c+1}^d x'_i \leq |\z| + (d-c)\cdot l\]
which contradicts the inequality $|\z|+(d-c)\cdot l  + 1\leq p\cdot d$.

``$\subseteq$": Let $\z\in\Z_p^d$, let $c=z_1$, and consider any partition $\x\in\X_p^d$ with $\x(c)\leq\z$ and $x'_{c+1}=l+1$. We get $x_1\geq c+1$, and since $x_1\leq d$ we conclude that $c\leq d-1$. The fact that $z_1=\cdots=z_{l+1}$ follows from Remark~\ref{rem:z1--zl+1}. Since $x'_i\leq l+1$ for $i=c+1,\cdots,d$ and $x'_i=0$ for $i>d$ we get
\[p\cdot d = |\x| = |\x(c)| + \sum_{i=c+1}^d x'_{c+1} \leq |\z| + (d-c)\cdot(l+1).\]
Suppose now that $|\z|+(d-c)\cdot l \geq p\cdot d$. As in the proof of ``$\supseteq$", we can find a partition $\y\in\P_l$ with $|\y|=p\cdot d - |\z|$ and $y_i\leq d-c$ for $i=1,\cdots,l$. Defining $\x\in\P_n$ via
\[x_i=z_i+y_i\rm{ for }i=1,\cdots,l,\rm{ and }x_i=z_i\rm{ for }i=l+1,\cdots,n,\]
we obtain a partition $\x\in\X_p^d$ satisfying $\x(c)\leq\z$ and $x'_{c+1}\leq l$, which contradicts property (2) in Definition~\ref{def:ZX}. It follows that we must have $|\z|+(d-c)\cdot l  + 1\leq p\cdot d$.

Finally, suppose that $l\geq p$. Since $z_1=\cdots=z_{l+1}=c$, we get $|\z|\geq c\cdot(l+1)$, so
\[p\cdot d \geq |\z| + (d-c)\cdot l + 1 \geq c\cdot(l+1) + (d-c)\cdot l + 1 = d\cdot l + c + 1\geq d\cdot p + 1,\]
which is a contradiction. It follows that $l\leq p-1$, concluding the proof.
\end{proof}

The relationship between regularity and the optimization problem from Section~\ref{sec:optimization} is given by the following.

\begin{lemma}\label{lem:maxregJzl=Rlpnd}
 For each $l=0,\cdots,p-1$ we have using the notation (\ref{eq:def-Rlpnd}) the equality
 \begin{equation}\label{eq:maxregJzl=Rlpnd}
 \max \{\reg(J_{\z,l}):\z\in\P_n\rm{ and }(\z,l)\in \Z_p^d\} = R_{l,p,n,d}.
 \end{equation}
\end{lemma}

\begin{proof} We recall the notation (\ref{eq:defTlz}), (\ref{eq:YU-lpnd}), (\ref{eq:def-glpnd}), as well as the equality (\ref{eq:regJzl}), and we make a change of variables as follows. If $\z\in\P_n$ is such that $z_1=\cdots=z_{l+1}$, we let $\y\in\P_{n-l}$ denote the partition defined by 
\begin{equation}\label{eq:yz}
y_i=z_{i+l}\rm{ for }i=1,\cdots,n-l.
\end{equation}
If $\t\in\T_l(\z)$ then we define $\u\in\P_{n-l}$ via 
\begin{equation}\label{eq:ut}
u_i=t_{n-l+1-i}\rm{ for }i=1,\cdots,n-l.
\end{equation}
We claim that the following equivalence holds
\begin{equation}\label{eq:equiv-zt-yu}
 (\z,l)\in\Z_p^d\rm{ and }\t\in\T_l(\z) \Longleftrightarrow (\y,\u)\in\YU(l,p,n,d)
\end{equation}
and that under the assumption that the equivalent conditions in (\ref{eq:equiv-zt-yu}) are satisfied, we have the equality
\begin{equation}\label{eq:g-from-zt}
g_{l,p,n,d}(\y,\u)=|\z| + |\t| - l - f_{l}(\z,\t).
\end{equation}
Once these are verified, the equality (\ref{eq:maxregJzl=Rlpnd}) follows since
\[\max \{\reg(J_{\z,l}):\z\in\P_n\rm{ and }(\z,l)\in \Z_p^d\} \overset{(\ref{eq:regJzl})}{=} \max\{|\z| + |\t| - l - f_{l}(\z,\t) : \z\in\P_n,\ (\z,l)\in\Z_p^d,\ \t\in\T_l(\z)\} = \]
\[\overset{(\ref{eq:equiv-zt-yu}-\ref{eq:g-from-zt})}{=} \max\{g_{l,p,n,d}(\y,\u):(\y,\u)\in\YU(l,p,n,d)\} \overset{(\ref{eq:def-Rlpnd})}{=} R_{l,p,n,d}.\]

By Remark~\ref{rem:z1--zl+1}, any $(\z,l)\in\Z_p^d$ satisfies $z_1=\cdots=z_{l+1}$. Using (\ref{eq:yz}), we have $|\z|-l\cdot z_1 = |\y|$, so the inequality $|\z|+(d-z_1)\cdot l + 1 \leq p\cdot d$ in (\ref{eq:def-Z_p^d}) is equivalent to $|\y|\leq d\cdot(p-l)-1$ in (\ref{eq:YU-lpnd}).  Moreover, since $y_1=z_{l+1}=z_1$, we get $|\z|-(l+1)\cdot z_1 = |\y|-y_1$, so the inequality $p\cdot d \leq |\z| + (d-z_1)\cdot(l+1)$ in (\ref{eq:def-Z_p^d}) is equivalent to $|\y|-y_1\geq d\cdot(p-1-l)$ in (\ref{eq:YU-lpnd}). The inequalities $t_{i+1}-t_i\leq z_{n-i}-z_{n-i+1}$ in (\ref{eq:defTlz}), for $i=1,\cdots,n-l-1$, are equivalent to $u_i-u_{i+1}\leq y_i-y_{i+1}$ in (\ref{eq:YU-lpnd}), for $i=1,\cdots,n-l-1$. Since $t_{n-l+1}=l$ and $z_l=z_{l+1}$, the inequality $t_{n-l+1}-t_{n-l}\leq z_l - z_{l+1}$ is equivalent to $u_1=t_{n-l}=l$. The condition that $z_1\leq d-1$ in (\ref{eq:def-Z_p^d}) is automatic, since it is implied by  $|\z|+(d-z_1)\cdot l  + 1\leq |\ul{z}|+(d-z_1)\cdot (l+1)$. This equivalences prove (\ref{eq:equiv-zt-yu}).

It remains to verify (\ref{eq:g-from-zt}) under the assumption that the conditions in (\ref{eq:equiv-zt-yu}) are satisfied. Since $z_l=z_{l+1}$ and $t_{n-l+1}=t_{n-l}=l$ as explained in the previous paragraph, we get
\[f_l(\z,\t) \overset{(\ref{eq:def-flzt})}{=} \sum_{i=1}^{n-l} t_i\cdot (z_{n-i} - z_{n+1-i} - t_{i+1} + t_i) = \sum_{i=1}^{n-l-1} t_i\cdot (z_{n-i} - z_{n+1-i} - t_{i+1} + t_i). \]
Using the change of variables (\ref{eq:yz}--\ref{eq:ut}), and exchanging $i\leftrightarrow n-l-i$, we can rewrite the above equality as
\[f_l(\z,\t) = \sum_{i=1}^{n-l-1} u_{i+1}\cdot((y_i-y_{i+1}) - (u_i-u_{i+1})).\]
To prove (\ref{eq:g-from-zt}) we then need to verify that
\begin{equation}\label{eq:zt-l=y+u}
|\z| + |\t| - l = l\cdot y_1 + |\y| + |\u|.
\end{equation}
Since $z_1=\cdots=z_l=y_1$, we get from (\ref{eq:yz}) that $|\z|=l\cdot y_1 + |\y|$. Since $t_{n-l+1}=l$, we get from (\ref{eq:ut}) that $|\t|-l=|\u|$. This yields (\ref{eq:zt-l=y+u}) and concludes our proof.
\end{proof}

\begin{corollary}\label{cor:reg-pows}
 For every $1\leq p\leq n$ and $d\geq 1$, we have
 \[\reg(S/I_p^d) = \max_{l=0,\cdots,p-1} R_{l,p,n,d},\ \reg(S/(I_p^d)^{sat}) = \max_{l=1,\cdots,p-1} R_{l,p,n,d},\rm{ and }\reg(S/I_p^{(d)}) = R_{p-1,p,n,d}.\]
 In particular, $\reg(I_p^d) \geq \reg((I_p^d)^{sat}) \geq \reg(I_p^{(d)})$.
\end{corollary}

\begin{proof}
 We write $\X=\X_p^d$ so that $I_p^d = I_{\X}$. It follows from (\ref{eq:regIX}) that
 \[\reg(S/I_{\X}) = \max_{(\z,l)\in\Z(\X)}\reg(J_{\z,l}) \overset{\rm{Lemma}~\ref{lem:maxregJzl=Rlpnd}}{=} \max_{l=0,\cdots,p-1} R_{l,p,n,d}. \]
 
 We have using Lemma~\ref{lem:saturation} and (\ref{eq:p-saturation}) that $(I_p^d)^{sat} = I_p^d : I_1^{\infty} = I_{\X^{:1}}$ and $I_p^{(d)} = I_p^d : I_{p-1}^{\infty} = I_{\X^{:(p-1)}}$. Since $\Z(\X^{:1}) = \{(\z,l)\in\Z(\X):l\geq 1\}$ and $\Z(\X^{:(p-1)}) = \{(\z,l)\in\Z(\X):l\geq p-1\} =  \{(\z,l)\in\Z(\X):l = p-1\}$, the remaining equalities follow by the same reasoning from Lemma~\ref{lem:maxregJzl=Rlpnd}.
\end{proof}

\begin{lemma}\label{lem:reg-I1-In}
 If $p=1$ or $p=n$ then $\reg(I_p^d) = \reg((I_p^d)^{sat}) = \reg(I_p^{(d)}) = p\cdot d$ for all $d\geq 1$.
\end{lemma}

\begin{proof}
By Lemma~\ref{lem:Zsymb-vs-Zpow}, it suffices to check that $\reg(I_p^{(d)}) = p\cdot d$ when $p=1$ or $p=n$. We note that in the case when $p=n$ the result is a direct consequence of Corollary~\ref{cor:reg-pows} and Lemma~\ref{lem:pl=nn-1}: indeed, $\reg(I_{n}^{(d)}) = \reg(S/I_{n}^{(d)}) + 1 = R_{n-1,n,n,d} + 1 = n\cdot d$. By the same argument, the case $p=1$ follows once we verify that $R_{0,1,n,d} = d-1$ for all $d\geq 1$. This follows from the fact that
\[\YU(0,1,n,d) \overset{(\ref{eq:YU-lpnd})}{=} \{ (\y,\ul{0}): \y \in \P_n, |\y|\leq d-1\},\]
and from the equality $g_{0,1,n,d}(\y,\ul{0}) \overset{(\ref{eq:def-glpnd})}{=} |\y|$, which is maximized when $|\y|=d-1$.
\end{proof}

\begin{lemma}\label{lem:Rpnn2}
 If $1<p<n$ and $1\leq d\leq n-2$ then $R_{p-1,p,n,d} \geq p\cdot d$.
\end{lemma}

\begin{proof}
 We define the partitions $\y,\u\in\P_{n-p+1}$ via $y_1=d-1$, $u_1=p-1$, $y_i=0$ and $u_i = \max(1,p-d)$ for $i=2,\cdots,n-p+1$. We have $y_1-y_2 = d-1\geq\min(p-2,d-1)=u_1-u_2$. Moreover $y_i-y_{i+1}=u_i-u_{i+1}=0$ for $i=2,\cdots,n-p$, and it follows from (\ref{eq:YU-lpnd}) that $(\y,\u)\in\YU(p-1,p,n,d)$. We have $(p-1)\cdot y_1 + |\y| = p\cdot(d-1)$, and one of the following:
 \begin{itemize}
  \item $p\leq d\leq n-2$, in which case $|\u|=(p-1)+(n-p)=n-1$ and
  \[\sum_{i=1}^{n-p} u_{i+1}\cdot((y_i-y_{i+1}) - (u_i - u_{i+1})) = (d-1) - ((p-1) - 1)=d-p+1.\]
 It follows that
 \[g_{p-1,p,n,d}(\y,\u) = p\cdot(d-1) + (n-1) + (d-p+1) = p\cdot d + (n-p-1) + (d-p+1) \geq p\cdot d.\]
  \item $1\leq d\leq p-1$, in which case $|\u|=(p-1)+(n-p)\cdot(p-d)\geq p$. Since $y_i-y_{i+1}=u_i-u_{i+1}$ for all $i=1,\cdots,n-p$, we conclude that $g_{p-1,p,n,d}(\y,\u)=(p-1)\cdot y_1 + |\y| + |\u|\geq p\cdot(d-1) + p = p\cdot d$.
 \end{itemize}
 Since $R_{p-1,p,n,d}\geq g_{p-1,p,n,d}(\y,\u)$, we see that in both cases $R_{p-1,p,n,d}\geq p\cdot d$, as desired.
\end{proof}

\begin{proof}[Proof of Theorem~\ref{thm:reg-pows}]
 The equality (\ref{eq:reg-powers-p=1n}) is proved in Lemma~\ref{lem:reg-I1-In}. The equalities in (\ref{eq:reg-large-powers}) follow by combining the fact that $\reg(I) = \reg(S/I) + 1$ with Corollary~\ref{cor:reg-pows} and Proposition~\ref{prop:Rlpnd-dlarge}. Finally, (\ref{eq:reg-small-powers}) follows in the same way from Corollary~\ref{cor:reg-pows} and Lemma~\ref{lem:Rpnn2}.
\end{proof}

We next give an explicit formula for $\Z(\X_p^{(d)})$, which shows that the modules $\Ext^{\bullet}_S(S/I_p^{(d)},S)$ grow with $d$.

\begin{lemma}\label{lem:Z_p^(d)}
If we let $\Z_p^{(d)} = \Z(\X_p^{(d)})$ then
\begin{equation}\label{eq:def-Z_p^(d)}
\Z_p^{(d)} = \{(\z,p-1): \z\in\P_n, z_1=\cdots=z_p, z_p + z_{p+1} + \cdots + z_n \leq d-1\}
\end{equation}
\end{lemma}

\begin{proof}
 Since $I_{\X_p^{(d)}} = I_{\X_p^{d}}:I_{p-1}^{\infty} = I_{(\X_p^{d})^{:(p-1)}}$ it follows from Remark~\ref{rem:ZdependsonI} that $\Z(\X_p^{(d)}) = \Z((\X_p^{d})^{:(p-1)})$, and therefore using (\ref{eq:Z-saturation}) we obtain
 \[\Z(\X_p^{(d)}) = \{(\z,l)\in\Z(\X_p^d) : l\geq p-1\} = \{(\z,l)\in\Z(\X_p^d) : l = p-1\}.\]
 From (\ref{eq:def-Z_p^d}) it follows that $(\z,p-1)\in\Z(\X_p^d)$ if and only if
 \begin{equation}\label{eq:zp-1-in-Zpd}
 z_1=\cdots=z_p \leq d-1,\ |\z| + (d-z_1)\cdot(p-1) + 1 \leq p\cdot d \leq |\z| + (d-z_1)\cdot p.
 \end{equation}
 If $z_1=\cdots=z_p$, then the inequality $p\cdot d \leq |\z| + (d-z_1)\cdot p$ is trivially satisfied since it is equivalent to 
 \[|\z|-p\cdot z_1 = z_{p+1}+\cdots+z_n\geq 0.\]
 Moreover, the inequality $|\z| + (d-z_1)\cdot(p-1) + 1 \leq p\cdot d$ is equivalent to $z_p+\cdots+z_n=|\z|-(p-1)\cdot z_1\leq d-1$, so (\ref{eq:zp-1-in-Zpd}) is equivalent to
 \[ z_1=\cdots=z_p \leq d-1,\rm{ and }z_p+z_{p+1}+\cdots+z_n\leq d-1.\]
 Since the inequality $z_p\leq d-1$ is implied by $z_p+z_{p+1}+\cdots+z_n\leq d-1$, (\ref{eq:def-Z_p^(d)}) follows.
\end{proof}

Since $\Z_p^{(d)} \subset \Z_p^{(d+1)}$ for all $d$, an immediate application of (\ref{eq:kerExt}) in Theorem~\ref{thm:Ext-split-IX} yields
\begin{corollary}\label{cor:inj-Ext-symb}
 For $1\leq p\leq n$, and $d,j\geq 0$, the surjection $S/I_p^{(d+1)}\onto S/I_p^{(d)}$ induces an inclusion
 \[\Ext^j_S(S/I_p^{(d)},S) \lhra \Ext^j_S(S/I_p^{(d+1)},S).\]
 In particular, for all $j\geq 0$ the induced maps $\Ext^j_S(S/I_p^{(d)},S)\lra H_{I_p}^j(S)$ are injective.
\end{corollary}

We can now give an answer to \cite[Question~6.2]{EMS} in the case of determinantal thickenings.

\begin{corollary}\label{cor:unmixed}
 If $I\subseteq S$ is a proper $\GL$-invariant ideal which is unmixed, then for each $j\geq 0$ the induced map
 \[\Ext^j_S(S/I,S) \lra H^j_I(S)\]
 is injective.
\end{corollary}

\begin{proof} Since $I$ is $\GL$-invariant, its radical is $\sqrt{I}=I_p$ for some $1\leq p\leq n$. The condition of $I$ being unmixed is then equivalent to $I=I:I_{p-1}^{\infty}$. If we write $I=I_{\X}$ then it follows from (\ref{eq:Z-saturation}) that if $(\z,l)\in\Z(\X)$ then $l\geq p-1$. Moreover, it follows from Corollary~\ref{cor:Ann-Jzl} that for every $(\z,l)\in\Z(\X)$ we must have $l\leq p-1$, since $J_{\z,l}$ is a subquotient of $S/I_{\X}$, and $S/I_{\X}$ is set-theoretically supported on rank $<p$ matrices. We conclude using Remark~\ref{rem:z1--zl+1} that $\Z(\X)$ is a finite subset of
\[\mc{A}=\{(\z,p-1):\z\in\P_n,\ z_1=\cdots=z_p\}.\]
It follows from (\ref{eq:def-Z_p^(d)}) that $\mc{A}=\cup_{d\geq 1}\Z_p^{(d)}$, so for $d\gg 0$ we must have $\Z(\X)\subseteq\Z_p^{(d)}$.

Since $\sqrt{I}=I_p$, there exists $d\gg 0$ such that $I\supseteq I_p^d$, and therefore $I = I:I_{p-1}^{\infty} \supseteq I_p^d : I_{p-1}^{\infty} = I_p^{(d)}$. If $d$ is chosen so that $I\supseteq I_p^{(d)}$ and $\Z(\X)\subseteq\Z_p^{(d)}$, then we have by (\ref{eq:kerExt}) that for each $j\geq 0$ the induced maps
\[\Ext^j_S(S/I,S) \lra \Ext^j_S(S/I_p^{(d)},S)\]
are injective. Since $\sqrt{I}=I_p$, we have $H^j_I(S) = H^j_{I_p}(S)$, and the conclusion follows from Corollary~\ref{cor:inj-Ext-symb}.
\end{proof}



We end this section with a verification of (\ref{eq:reg-I2}), for which we need to analyze the regularity of small powers of $2\times 2$ minors.

\begin{theorem}\label{thm:2x2}
 Assume that $m\geq n\geq 3$ and let $I_2$ denote the ideal of $2\times 2$ minors of the generic $m\times n$ matrix. We have
 \[\reg(I_2^d) = \reg(I_2^{(d)}) = d + n - 1\rm{ for }1\leq d\leq n-2.\]
\end{theorem}

\begin{proof} By Corollary~\ref{cor:reg-pows} we have
\[\reg(I_2^{(d)}) = 1 + R_{1,2,n,d}\rm{ and }\reg(I_2^d) = 1 + \max(R_{0,2,n,d},R_{1,2,n,d}).\]
In order to compute $R_{0,2,n,d}$ and $R_{1,2,n,d}$, we recall the notation (\ref{eq:YU-lpnd}--\ref{eq:def-Rlpnd}). We have
\[\YU(0,2,n,d) = \{(\y,\ul{0}) : \y\in\P_n, |\y| \leq 2d - 1, |\y|-y_1\geq d\},\]
and $g_{0,2,n,d}(\y,\ul{0}) = |\y| \leq 2d-1$. When $d\geq 2$, the value of $g_{0,2,n,d}(\y,\ul{0})$ is then maximized for $\y=(d-1,d-1,1,0,\cdots)$ and we obtain $R_{0,2,n,d}=2d-1$, while for $d=1$ we get $\YU(0,2,n,1)=\emptyset$ and $R_{0,2,n,1}=-\infty$. It follows that $R_{0,2,n,d} < d + n - 2$, for $1\leq d\leq n-2$. To prove the theorem, it is then enough to verify that $R_{1,2,n,d}=d+n-2$ for $1\leq d\leq n-2$. We have
\[\YU(1,2,n,d) = \{(\y,\u)\in\P_{n-1}\times\P_{n-1}: |\y|\leq d-1,u_1=1,y_i-y_{i+1} \geq u_i-u_{i+1}\rm{ for }1\leq i\leq n-2\}.\]
Consider $(\y,\u)\in\YU(1,2,n,d)$ and let $1\leq r\leq n-1$ such that $u_1=\cdots=u_r = 1$, $u_{r+1}=0$. We get
\[g_{1,2,n,d}(\y,\u) = y_1 + |\y| + r - (y_1-y_r - u_1 + u_r) = |\y| + r + y_r.\]
Taking $(\y,\u)=((1^{d-1}),(1^{n-1}))\in\YU(1,2,n,d)$, we get $r=n-1>d-1$ and thus $y_r=0$. We obtain 
\[R_{1,2,n,d}\geq g_{1,2,n,d}(1^{d-1},1^{n-1}) = d-1 + n - 1 = d+n-2.\]
For the reverse inequality, we have to verify that $|\y|+r+y_r \leq d+n-2$ when $(\y,(1^r))\in\YU(1,2,n,d)$. Since $|\y|\leq d-1$, it is enough to check that $r+y_r\leq n-1$. Supposing instead that $y_r\geq n-r$, we obtain 
\[d-1\geq |\y| \geq r\cdot y_r\geq r\cdot(n-r)\geq r+(n-r) - 1 = n-1,\]
which contradicts the inequality $d\leq n-2$ and concludes our proof.
\end{proof}

\section{Kodaira Vanishing for Thickenings}\label{sec:Kodaira}

In \cite[Section~3]{BBLSZ}, the authors prove a version of the Kodaira vanishing theorem for the thickenings of local complete intersections which are defined by a power of the ideal sheaf (and show that the statement is false for more general thickenings). In this section we prove that their vanishing result holds for arbitrary $\GL$-equivariant thickenings of determinantal varieties.

\begin{theorem}\label{thm:Kodaira}
If $Y\subset\bb{P}^{m\cdot n-1}$ is a $\GL$-equivariant determinantal thickening then
 \[H^k(Y,\mc{O}_{Y}(-j))=0\rm{ for }k<m+n-2\rm{ and }j>0.\]
In particular, if we let $Y_{\operatorname{red}}$ denote the underlying determinantal variety, and if we make the convention that $\codim(\Sing(Y_{\operatorname{red}}))=\dim(Y_{\operatorname{red}})$ when $Y_{\operatorname{red}}$ is non-singular, then
 \[H^k(Y,\mc{O}_{Y}(-j))=0\rm{ for }k<\codim(\Sing(Y_{\operatorname{red}}))\rm{ and }j>0.\]
\end{theorem}
 
\begin{proof}
 We write $\mf{m}$ for the maximal homogeneous ideal in the polynomial ring $S$, and $H^{\bullet}_{\mf{m}}$ for the corresponding local cohomology groups. If $I\subseteq S$ denotes the homogeneous ideal of $Y$, then we have a degree preserving exact sequence
 \[S/I \lra \bigoplus_{j\in\bb{Z}} H^0(Y,\mc{O}_{Y}(-j)) \lra H^1_{\mf{m}}(S/I_p^d)\]
 and isomorphisms $H^k(Y,\mc{O}_{Y}(-j)) = H^{k+1}_{\mf{m}}(S/I)_{-j}$ for all $k>0$ and $j\in\bb{Z}$. Since $(S/I)_{-j}=0$ for $j>0$, and since by graded local duality \cite[Thm.~3.6.19]{bruns-herzog} we have isomorphisms of finite dimensional vector spaces
 \[H^{k+1}_{\mf{m}}(S/I)_{-j} \simeq \Ext^{m\cdot n-1-k}_S(S/I,S)_{-m\cdot n + j}\rm{ for all }k,\]
 in order to prove the desired vanishing statement it is enough to check by Lemma~\ref{lem:Z_p^d} and (\ref{eq:ExtS/IX}) that
 \begin{equation}\label{eq:vanishing}
 \Ext^{m\cdot n-1-k}_S(J_{\z,l},S)_j =  0\rm{ for }j>-m\cdot n\rm{ and }k < m+n-2.
 \end{equation}
 Suppose that $\ll\in\bb{Z}^n_{\dom}$ is a dominant weight of size $|\ll| = j > -m\cdot n$ and that $S_{\ll(s)}\bb{C}^m\oo S_{\ll}\bb{C}^n$ appears as a subrepresentation of $\Ext^{m\cdot n-1-k}_S(J_{\z,l},S)_j$, where we use the notation in (\ref{eq:Extj}). It follows that there exist $0\leq s\leq t_1\leq\cdots\leq t_{n-l}\leq l$ such that $\ll\in W(\z,l;\t,s)$ and
 \begin{equation}\label{eq:ineq-k}
 m\cdot n - 1 - k = m\cdot n - l^2 - s\cdot(m-n) - 2\cdot \left(\sum_{i=1}^{n-l}t_i \right)
 \end{equation}
 If $s=0$, then the condition (\ref{eq:lam-in-W}) implies that $\ll_1 \leq -m$ and thus $|\ll| \leq -m\cdot n$, a contradiction. It follows that $s\geq 1$, and therefore $l\geq 1$ and $t_i\geq 1$ for all $i=1,\cdots,n-l$. We get
 \[m\cdot n - l^2 - s\cdot(m-n) - 2\cdot \left(\sum_{i=1}^{n-l}t_i \right) \leq m\cdot n - l^2 - (m-n) - 2\cdot(n-l) \leq m\cdot n - (m+n-1).\]
 Using (\ref{eq:ineq-k}) we find that $k+1\geq m+n-1$, i.e. $k\geq m+n-2$, which proves (\ref{eq:vanishing}).
 
 To prove the last statement of the theorem we note that $Y_{\operatorname{red}}$ is non-singular only if its defining ideal is $I_2$ (i.e. if it is isomorphic to the Segre product $\bb{P}^{m-1}\times\bb{P}^{n-1}$), in which case $\dim(Y) = m+n-2$. If the defining ideal of $Y_{\operatorname{red}}$ is $I_p$ with $p\geq 3$, then $\codim(\Sing(Y_{\operatorname{red}})) = m+n-2p+3<m+n-2$.
\end{proof}

\section{An Example}\label{sec:example}

We conclude our paper with a concrete calculation which recovers the table in \cite[Example~5.4]{BBLSZ}. We let $m=n=3$, and consider the ideal $I_2$ of $2\times 2$ minors of the generic $3\times 3$ matrix. We are interested in computing the modules $\Ext^9(S/I_2^d,S)$ for $d\geq 1$. It follows from (\ref{eq:ExtS/IX}) and Lemma~\ref{lem:Z_p^d} that
\[\Ext^9_S(S/I_2^d,S) = \bigoplus_{(\z,l)\in\Z_2^d} \Ext^9_S(J_{\z,l},S),\]
Using Theorem~\ref{thm:ExtJzl}, we see that $\Ext^9_S(J_{\z,l},S)=0$ unless $l=0$, in which case
\begin{equation}\label{eq:Ext9Jz0}
\Ext^9_S(J_{\z,0},S) = \bigoplus_{\ll\in W(\z,0;\ul{0},0)}S_{\ll}\bb{C}^3\oo S_{\ll}\bb{C}^3,
\end{equation}
where $W(\z,0;\ul{0},0)=\{\ll\}$ is a singleton, consisting of the weight $\ll = (-3-z_3,-3-z_2,-3-z_1)$, and the term $S_{\ll}\bb{C}^3\oo S_{\ll}\bb{C}^3$ appears in degree $|\ll|=-|\z|-9$. The following table records the partitions $\z\in\P_3$ for which $(\z,0)\in\Z_2^d$ when $d\leq 7$: by (\ref{eq:def-Z_p^d}) we must have $|\z|\leq 2d-1$ and $z_2+z_3\geq d$. We use $||$ to separate the various partitions $\z$ according to their size.

\begin{center}
\setlength{\extrarowheight}{2pt}
\begin{tabular}{c|c}
$d$ & $\z$ \\
\hline
1 & \\
\hline
2 & $(1^3)$ \\
\hline
3 & $(2^2,1)$ \\
\hline
4 & $(2^3)\ ||\ (3,2^2),\ (3^2,1)$ \\
\hline
5 & $(3^2,2)\ ||\ (3^3),\ (4^2,1),\ (4,3,2)$ \\
\hline
6 & $(3^3)\ ||\ (4^2,2),\ (4,3^2)\ ||\ (4^2,3),\ (5^2,1),\ (5,4,2),\ (5,3^2)$ \\
\hline
7 &  $(4^2,3)\ ||\ (4^3),\ (5^2,2),\ (5,4,3) \ ||\ (5^2,3),\ (5,4^2),\ (6^2,1),\ (6,5,2),\ (6,4,3)$ \\
\end{tabular}
\end{center}

We note that $\dim S_{\ll}\bb{C}^3 = \dim S_{\z}\bb{C}^3$ when $\ll = (-3-z_3,-3-z_2,-3-z_1)$, so writing $S_{\z}$ instead of $S_{\z}\bb{C}^3$ we get for instance
\[
\begin{aligned}
\dim\left(\Ext^9_S(S/I_2^7,S)_{-22}\right) =\ & (\dim S_{5^2,3})^2 + (\dim S_{5,4^2})^2 + (\dim S_{6^2,1})^2 + (\dim S_{6,5,2})^2 + (\dim S_{6,4,3})^2 \\ 
=\ & 6^2 + 3^2 + 21^2 + 24^2 + 15^2 = 1287.
\end{aligned}
\] 
Since the partition $\z=(4^2,3)$ is the only one with the property that $(\z,0) \in \Z_2^6\cap\Z_2^7$, we get from (\ref{eq:imExt}) that
\[\rm{Im}\left(\Ext^9_S(S/I_2^6,S) \lra \Ext^9_S(S/I_2^7,S)\right) = \Ext^9_S(J_{(4^2,3),0},S) \overset{(\ref{eq:Ext9Jz0})}{=} S_{(-6,-7,-7)}\bb{C}^3 \oo S_{(-6,-7,-7)}\bb{C}^3,\]
which coincides with the degree $-20$ component of $\Ext^9_S(S/I_2^7,S)$ and is $9$-dimensional.

\section*{Acknowledgments} 
We would like to thank Anurag Singh for sharing the preprint \cite{BBLSZ} and for the interesting discussions that started this project. We are also grateful to Michael Perlman for reading the manuscript and for pointing out an omission in an earlier proof of Proposition~\ref{prop:Rlpnd-dlarge}. Experiments with the computer algebra software Macaulay2 \cite{M2} have provided numerous valuable insights. The author acknowledges the support of the Alfred P. Sloan Foundation, and of the National Science Foundation Grant No.~1600765.


\end{document}